\documentclass[12pt,reqno]{amsart}
\usepackage{amssymb}
\usepackage{amsmath}
\usepackage{enumitem}
\usepackage{mathrsfs}
\usepackage[all]{xy}
\usepackage{hyperref}
\usepackage{marginnote}

\usepackage{stmaryrd}

\usepackage[margin=1.25in]{geometry}

\RequirePackage{mathrsfs} 

\numberwithin{equation}{section}

\newtheorem{theorem}[equation]{Theorem}
\newtheorem{lemma}[equation]{Lemma}
\newtheorem{proposition}[equation]{Proposition}
\newtheorem{corollary}[equation]{Corollary}

\theoremstyle{definition}
\newtheorem*{ack}{Acknowledgements}
\newtheorem*{con}{Conventions}
\newtheorem{remark}[equation]{Remark}
\newtheorem{example}[equation]{Example}
\newtheorem{definition}[equation]{Definition}

 \numberwithin{figure}{section}

\DeclareMathOperator{\Aut}{Aut}

\DeclareMathOperator{\Spec}{Spec}

\DeclareMathOperator{\an}{an}

\DeclareMathOperator{\Hom}{Hom}

\DeclareMathOperator{\Br}{Br}

 \DeclareMathOperator{\Out}{Out}
 \DeclareMathOperator{\Isom}{Isom}
\DeclareMathOperator{\Ideg}{Ideg}
\let\Im\relax
\DeclareMathOperator{\Im}{Im}
\DeclareMathOperator{\lien}{band}

\newcommand{\Qbar}{\overline{\QQ}}

\newcommand{\thickslash}{\mathbin{\!\!\pmb{\fatslash}}}

\newcommand\ZZ{\mathbb{Z}}
\newcommand\NN{\mathbb{N}}
\newcommand\QQ{\mathbb{Q}}

\newcommand\CC{\mathbb{C}}
\newcommand\GG{\mathbb{G}}

\newcommand\Gm{\GG_\mathrm{m}}
\newcommand\OO{\mathcal{O}}
\newcommand\HH{\mathrm{H}}

\usepackage{color}

\newcommand{\dan}[1]{{\color{blue} \sf $\clubsuit\clubsuit\clubsuit$ Dan: [#1]}}

\title[Arithmetic hyperbolicity]{Arithmetic hyperbolicity and a stacky Chevalley--Weil theorem}
\author{Ariyan Javanpeykar}
\address{Ariyan Javanpeykar \\
Institut f\"{u}r Mathematik\\
Johannes Gutenberg-Universit\"{a}t Mainz\\
Staudingerweg 9, 55099 Mainz\\
Germany.}
\email{peykar@uni-mainz.de}

\author{Daniel Loughran}
\address{Daniel Loughran \\
Department of Mathematical Sciences \\
University of Bath \\
Claverton Down \\
Bath \\
BA2 7AY \\
UK.}



\begin{document}

\begin{abstract}
 We prove  an analogue for  algebraic stacks of Hermite--Minkowski's finiteness theorem from algebraic number theory, and establish   a Chevalley--Weil type theorem for integral points on stacks.
As an application of our results,  we  prove analogues of the Shafarevich conjecture for some surfaces of general type.
\end{abstract}

\maketitle
\tableofcontents

\thispagestyle{empty}

\section{Introduction}
Algebraic stacks are an important tool in modern algebraic geometry which naturally arise in the study of moduli problems.
In this paper we study various arithmetic properties of stacks defined over number fields and their rings of integers.
We obtain analogues for stacks of some classical theorems in algebraic number theory and arithmetic geometry, and  give 
  applications to the study of integral points on various moduli stacks.

\subsection{Hermite--Minkowski for stacks}
A classical result in algebraic number theory is the theorem of Hermite--Minkowski. A geometric
way to phrase this is:~given a positive integer $n$, a number field $K$, and a dense open subset $U \subset \Spec \OO_K$, the scheme $U$ admits only finitely many
isomorphism classes of finite \'{e}tale covers of degree at most $n$ (versions of this are also known for $\ZZ$-finitely generated subrings of $\CC$ \cite{smallness}).

Our first theorem is a generalisation of this result to algebraic stacks. One has to be careful however with formulating the correct statement.
Firstly, as finite morphisms are  by definition representable, to get interesting stacks one needs to study \emph{proper} \'{e}tale morphisms. Secondly,  extra phenomena appear in the case of  stacks due to  their \emph{inertia groups} (see \S \ref{sec:degree}). For example, if $X\to Y$ is a proper \'etale morphism, then its degree $\deg(X\to Y$) is only a \emph{rational} number in general. Thirdly, bounding the degree is not sufficient:~for every prime power $p^n$, the map $B(\ZZ/\varphi(p^n)\ZZ) \to \Spec \ZZ[\mu_{p^n},1/p] \to \Spec \ZZ[1/p]$ is a proper \'etale morphism of degree $1$. One therefore also needs to bound what we call the \emph{inertia degree} $\Ideg(X/Y)$  (see \S \ref{sec:degree} for a precise definition). Bearing these considerations in mind, our result is the following.

\begin{theorem} [Hermite--Minkowski for stacks]\label{thm:HM_intro} 
Let $A\subset \mathbb C$ be a    $\ZZ$-finitely generated subring with fraction field $K:=\mathrm{Frac}(A)$, and let $n \in \NN$. Then the set of $K$-isomorphism classes of  proper \'{e}tale integral algebraic stacks ${X}$ over $A$ such that  $$\max\{\deg(X/A), \Ideg(X/A)\} \leq n$$  is finite.
\end{theorem}

Here we say that two such stacks $X/A$ and $Y/A$ are $K$-isomorphic if there is an isomorphism  $X_K \cong Y_K$ defined over $K$.
Theorem \ref{thm:HM_intro} is proved by reducing to two cases: finite \'{e}tale morphisms, where this is the usual Hermite--Minkowski, and the case of \emph{gerbes},
which we handle using non-abelian cohomology.

Note that Theorem \ref{thm:HM_intro} may be viewed as a common generalisation of the theorem of Hermite--Minkowski, and the fact that the $n$-torsion of  $\Br \OO_K[S^{-1}]$ is finite for any $n$ and any finite set of finite places $S$ (this follows by viewing Brauer group elements as gerbes). The latter finiteness is well-known from class field theory and can be deduced from the fundamental exact sequence \cite[Ex.~6.9.4]{Poo17} .

\subsection{Chevalley--Weil for stacks}
The Hermite--Minkowski theorem has many applications in number theory, and we expect our version for stacks to have similar applications.  
We give one such application here, which is a version of the Chevalley--Weil theorem for stacks; see \cite{CW} or \cite[\S 4.2] {Ser97} for a formulation of the classic Chevalley--Weil theorem.

\begin{theorem} [Chevalley--Weil for stacks]\label{thm:cv_intro} Let $A\subset \mathbb C$ be an integrally closed $\ZZ$-finitely generated subring, and let $Y$ be a finite type separated Deligne--Mumford   stack over $A$. Let $X\to Y$ be a  morphism such that $X_{\CC}\to Y_{\CC}$ is proper \'etale surjective.
Suppose that for every $\ZZ$-finitely generated  subring $A \subset B \subset \CC$ which is \'etale over $A$,
the groupoid $X(B)$ is finite. Then the groupoid    $Y(A)$ is finite.
\end{theorem}

In the classical version of Chevalley--Weil, one takes $A=\OO_{K}[S^{-1}]$, the ring of $S$-integers of some number field $K$ for some suitable large set of places $S$, and $X$ and $Y$ are models of some varieties over $\OO_K$. The theorem states that given $P \in Y(A)$, there exists a finite extension $K \subset L$ which is unramified outside of $S$ such that $P$ comes from an element of $X(L)$. Combining this with Hermite--Minkowski and assuming that $X$ has only finitely many $\OO_{L}[T^{-1}]$-integral points for every finite field extension $L/K$ and every finite set of finite places $T$ of $L$, one deduces that $Y(A)$ is finite.

Our stacky version (Theorem \ref{thm:cv_intro}) also concerns integral points and is the generalisation of this latter finiteness result. However, for stacks, $Y(A)$ is a groupoid in general, and not necessarily a set; thus one needs to quotient out by isomorphisms to get meaningful finiteness statements (we say that a groupoid $\mathcal{G}$ is \emph{finite} if the set of isomorphism classes of objects of $\mathcal{G}$ is finite).

Many stacks naturally arise as moduli stacks of varieties. If $Y$ has such a modular interpretation, then the finiteness of $Y(A)$ can be interpreted in terms of analogues of the \emph{Shafarevich conjecture}  on the finiteness of varieties  with good reduction outside a given set of places. (This was  originally proved by Faltings for curves and abelian varieties in \cite{Faltings2, FaltingsComplements}.)

To explain the connection between finiteness of integral points on stacks and the aforementioned Shafarevich conjecture, let $g \in \NN$ and let $\mathcal{A}_g$ be the   stack of principally polarised abelian schemes over $\ZZ$. For a number field $K$ and a finite set of places $S$ of $K$, an object $A$ in $  \mathcal{A}_g(\OO_{K}[S^{-1}])$ is a principally polarised abelian scheme over $\OO_{K}[S^{-1}]$; the corresponding abelian variety $A_K$ over $K$ therefore has good reduction outside of $S$.  It is easy to see that the essential image of the natural morphism of groupoids
\begin{equation} \label{Faltings}
	\mathcal{A}_g(\OO_{K}[S^{-1}]) \to \mathcal{A}_g(K)
\end{equation}
consists of exactly those principally polarised abelian varieties of dimension $g$ over $K$ with good reduction outside of $S$. In particular, an equivalent way to   state Faltings's finiteness theorem is that the image of \eqref{Faltings} is finite. (In fact the groupoid $\mathcal{A}_g(\OO_{K}[S^{-1}])$ itself is also finite, as the functor \eqref{Faltings} is fully faithful due to the properties of the N\'eron model of an abelian variety.) Thus, as this example shows, it is very natural to    study   integral points on stacks  in arithmetic geometry.

\subsection{Arithmetic hyperbolicity}  The stacky Chevalley--Weil theorem is part of a more general study in this paper of the phenomenon of \emph{arithmetic hyperbolicity} for stacks. A variety $X$ over an algebraically closed field $k$ of characteristic zero is called arithmetically hyperbolic over $k$ if,  for every $\ZZ$-finitely generated subring $A\subset k$ and every finite type separated scheme $\mathcal{X}$ over $A$ with $\mathcal{X}_k\cong X$, the set of $A$-valued points $\mathcal{X}(A)$ on $\mathcal{X}$ is finite. We extend this notion to stacks in \S \ref{section:arithmetic_hyperbolicity}. To state the more general definition, for $\mathcal{G}$ a (small) groupoid, we let $\pi_0(\mathcal{G})$ denote the set of isomorphism classes of objects of $\mathcal{G}$.   Then, a finitely presented algebraic stack $X$ over $k$ is \emph{arithmetically hyperbolic over $k$} if, for all $\ZZ$-finitely generated subrings $A\subset k$ and every finitely presented algebraic stack    $\mathcal X$  over $A$ endowed with an isomorphism $\mathcal{X}_k\cong X$, the set $\Im[\pi_0(\mathcal X(A))\to \pi_0(\mathcal X(k))]$ is finite; this is equivalent to Definition \ref{defn:arithmetic_hyperbolicity} by Lemma \ref{lem:arithm_hyp}. The Chevalley--Weil theorem for stacks (as formulated in Theorem \ref{thm:cv_intro}) is in fact a consequence of the following more general statement (stated as Theorem \ref{thm:chev_weil}).

\begin{theorem}\label{thm:chev_weil_intro}
	Let $f:X\to Y$ be a proper \'etale surjective morphism of finitely presented algebraic stacks over an algebraically closed field $k$ of characteristic $0$. Then $X$ is arithmetically hyperbolic over $k$ if and only if $Y$ is arithmetically hyperbolic over $k$.
\end{theorem}

Further properties of arithmetically hyperbolic varieties are obtained in \cite{Jaut, JLitt}, building on the results in \S \ref{section:arithmetic_hyperbolicity}.

\subsection{Applications}
As an application of our results, we obtain a finiteness result for certain surfaces of general type. By arguing directly with the stack, the proof of the following result becomes an application of classification results due to Beauville \cite[Appendix]{DebarreClass}  for such surfaces and  our stacky Chevalley--Weil theorem. 
 For a smooth proper surface $S$ over $\CC$, we denote by $q(S) = \dim_{\CC} \mathrm{H}^1(\mathcal{O}_S)$ its irregularity and $p_g(S) = \dim_{\CC} \mathrm{H}^0(K_S)$ its geometric genus.

\begin{theorem}\label{thm:shaf_for_gt} Let $A\subset \CC$ be an integrally closed $\ZZ$-finitely generated subring and let $q \geq 4$ be an integer. Then
the set of $A$-isomorphism classes of   smooth proper surfaces $X$ over $A$ with $\omega_{X/A}$ relatively ample, such that $q(X_{\CC}) = q$ and $p_g(X_\CC) = 2q(X_\CC) - 4$, is finite.  
\end{theorem}

Other applications of  the Chevalley--Weil theorem (Theorem \ref{thm:cv_intro}) arise in the case of moduli stacks which are proper \'etale gerbes over other moduli stacks. Such situations arise  quite naturally, and we give various       applications of the stacky  Chevalley--Weil  theorem to such moduli stacks in \cite{JLalg2}.

\subsection*{Outline of the paper} In \S\ref{sec:degree} we define and study the basic properties of the degree and inertia degree of a morphism of stacks, together with some properties of proper \'etale morphisms. In \S\ref{sec:HM} we prove Theorem \ref{thm:HM_intro} and in \S\ref{section:arithmetic_hyperbolicity} we develop the general theory of arithmetic hyperbolicity for algebraic stacks. Theorem \ref{thm:cv_intro} is proved     in \S\ref{sec:CW}. Moreover, in the same section we deduce Theorem \ref{thm:chev_weil_intro} from a more general result (see Theorem \ref{thm:chev_weil}).  Finally, \S\ref{sec:appplications} contains applications of our results, including the proof of Theorem \ref{thm:shaf_for_gt} and a purely ``transcendental'' criterion for a stack to be arithmetically hyperbolic (Theorem \ref{thm:criterion}).


\begin{ack}  We thank Johan Commelin for helpful discussions on \S \ref{section:gt}. We thank Andrew Kresch and Siddharth Mathur for many helpful discussions on gerbes, especially \S \ref{section:classifying_gerbes}. We are grateful to Kenneth Ascher, Kai Behrend,  Oliver Benoist, Yohan Brunebarbe,   Fr\'ed\'eric Campana,  Francois Charles,  Bas Edixhoven, Robin de Jong, Daniel Litt,  Behrang Noohi, Martin Olsson, David Rydh, and Angelo Vistoli for helpful discussions.  We thank the referee for a careful reading of our paper and useful comments. The first-named author gratefully acknowledges  support of   SFB/Transregio 45.  The second-named author is supported by EPSRC grant EP/R021422/1.

\end{ack}

\begin{con}
For a number field $K$, we let $\OO_K$ denote its ring of integers. If $S$ is a finite set of finite places of $K$, we let $\OO_K[S^{-1}]$ denote the ring of $S$-integers of $K$. 

If $k$ is a field, a \textit{variety over $k$} is a finite type separated $k$-scheme. 

An \emph{arithmetic scheme} is a finite type flat scheme over $\ZZ$.

	We abbreviate \emph{quasi-compact quasi-separated} as qcqs.

A  (small) groupoid $\mathcal G$ is finite if the set  $\pi_0(\mathcal G)$ of isomorphism classes of objects in $\mathcal{G}$ is a finite set. (Note that we do not ask for the automorphism group of an object of $\mathcal{G}$ to be finite.)

If $f:X\to Y$ is a  morphism of algebraic stacks, then we say that $f$ is quasi-finite if $f$ is finite type and locally quasi-finite (in the sense of \cite[Tag~06PU]{stacks-project}). If $f:X\to Y$ is quasi-finite,  for all geometric points $\overline{y} : \Spec \overline{k}\to Y$ of $Y$, the groupoid $X_{\overline{y}}(\overline{k})$ is finite.

  Concerning gerbes,  we follow the (standard) conventions of the stacks project \cite[Tag~06QB]{stacks-project}. 
If $G$ is a locally finitely presented flat group scheme over a scheme $S$, we let $BG = [S/G]$ be the classifying stack of $G$-torsors over $S$ (for the fppf topology). For a scheme $T$ over $S$, the objects of the groupoid $BG(T)$ are $G_T$-torsors over $T$.  In particular, there is a natural bijection $\pi_0(BG(T)) = \mathrm{H}^1(T,G_T)$. (All cohomology in this paper is taken with respect to the fppf topology). 

For $X$ a stack, we let $I_X$ be the (absolute) inertia stack of $X$. For $X\to Y$ a morphism of stacks, we let $I_{X/Y}$ be the relative inertia stack of $X\to Y$. We refer the reader to \cite[Tag~050P]{stacks-project} and \cite[Tag~04YX]{stacks-project} for precise definitions.

For an (abstract) group $G$ and a scheme $S$, we denote the associated constant group scheme over $S$ by $G_S$.

	Let $X$ be a finitely presented algebraic stack over a field $k$. Let $A\subset k $ be a subring. A \emph{model for $X$ over $A$} is a pair $(\mathcal X, \phi)$ with $\mathcal X$ a finitely presented algebraic stack over $A$ and $\phi: \mathcal X\times_A k\to X$ an isomorphism over $k$. We will usually omit $\phi$ from the notation.

A morphism $X \to Y$ of algebraic stacks is called \emph{representable} (resp.~\emph{strongly representable}) if for every scheme $Z$ and morphism $Z \to Y$, the fibre product $Z \times_Y X$ is an algebraic space (resp.~a scheme). We use slightly different terminology to the stacks project \cite[Tag 04XA]{stacks-project}.

A proper \'etale morphism is \emph{finite} \'etale if and only if it is strongly representable; see   \cite[Tag~0CHU]{stacks-project}.

\end{con}



\section{Degree and inertia degree of a morphism of stacks} \label{sec:degree}

 In this section we define and study the degree and inertia degree of a morphism of stacks. 

\subsection{Degree}
In \cite[Def.~1.15]{Vistoli} Vistoli defines the degree of a separated dominant morphism $X \to Y$ of finite type integral Deligne--Mumford   stacks. We explain how to generalise this definition.  

We let $f:X \to Y$ be a finitely presented Deligne--Mumford morphism of qcqs (= quasi-compact quasi-separated)  algebraic stacks with $Y$ integral. Let $V \to Y$ be a smooth dominant finitely presented morphism from an integral scheme $V$.

\subsubsection{$X,Y$ algebraic spaces} \label{sec:schemes}
We first recall the definition in the case where $X$ and $Y$ are algebraic spaces. Choose dense open subschemes $X^\circ \subset X$ and $ Y^\circ \subset Y$ (these exist by \cite[Tag 03JG]{stacks-project}). If $f$ is not generically finite then we define $\deg f = \infty$. Otherwise there exists an open dense subset $U \subset Y^{\circ}$ such that $f_{|U}$ is finite flat \cite[Tags 02NV, 0529]{stacks-project}. In this case, the sheaf $f_* \OO_X|_U$ is  locally free of finite rank on  $U$ \cite[Tag 02K9]{stacks-project}; we define $\deg f$ to be its rank. Note that $\deg f$ is a non-negative integer. In particular, the degree of a morphism to an integral scheme is  defined to be the degree of its generic fibre.

\subsubsection{$X \to Y$ representable} \label{sec:rep}
If $X \to Y$ is  representable, then the fibre product $X \times_Y V$ is an algebraic space over the scheme $V$.
Hence from \S\ref{sec:schemes} we may define 
$$\deg(X\to Y) :=\deg(X\times_Y V\to V).$$

\subsubsection{$X$ a DM stack, $Y$ a scheme} \label{sec:DM_scheme}
Next, assume that $Y$ is a scheme, so that $X$ is a Deligne--Mumford stack \cite[Tag 04YW]{stacks-project}. We assume further that $X$ is reduced.  To define the degree we sum over the irreducible components of $X$; in particular we may assume that $X$ is integral. Let $U \to X$ be a finitely presented \'etale dominant morphism with $U$ an integral scheme. Using \S\ref{sec:rep} we may then define
$$\deg(X\to Y) :=\frac{\deg(U \to Y)}{\deg(U \to X)},$$
as each morphism occurring on the right is  representable.

\subsubsection{$X \to Y$ a DM morphism} \label{sec:general}
We finally consider the general case with $X$  reduced. As $X \times_Y V$ is a Deligne--Mumford   stack over $V$ \cite[Tag 04YW]{stacks-project}, by \S\ref{sec:DM_scheme} we may define
$$\deg(X \to Y) = \deg( X \times_Y V \to V).$$

\subsubsection{Properties}
In the special case of a separated dominant morphism of finite type qcqs integral Deligne--Mumford   stacks, our definition recovers Vistoli's definition \cite[Def.~1.15]{Vistoli}.
A similar argument to \cite[Lem.~1.16]{Vistoli} shows that this definition is independent of our choices, and that given finitely presented Deligne--Mumford morphisms $X \to Y$ and $Y \to Z$ of reduced qcqs stacks with $Y$ and $Z$ integral, we have the multiplicative property 
\begin{equation} \label{eqn:degree_multiplicative}
	\deg(X \to Y)\deg(Y \to Z) = \deg(X \to Z).
\end{equation}
Note that the degree is preserved along any smooth dominant base change from an integral scheme.

\subsection{Inertia degree}
Our definition of inertia degree is the following.

\begin{definition}[Inertia degree]
Let $X \to Y$ be a finitely presented Deligne--Mumford morphism of integral qcqs algebraic stacks.  We define the \emph{inertia degree} of $X \to Y$ to be
	$$\Ideg(X \to Y) := \deg(I_{X/Y} \to X).$$
\end{definition}	
Here $I_{X/Y}$ denotes the relative inertia stack.
Note that $I_{X/Y} \to X$ is a finitely presented representable morphism, hence $\Ideg(X \to Y) \in \ZZ$ by \S\ref{sec:rep}. The inertia degree ``measures how far'' a morphism is from being  representable over some dense open substack of $Y$.


\subsection{Proper \'etale morphisms}
We now study the above notions of degree and inertia degree for a proper \'etale morphism $X\to Y$ of algebraic stacks.

\begin{lemma} \label{lem:inertia}
	Let $f:X \to Y$ be a proper \'etale morphism of   algebraic stacks.
	Then the relative inertia stack $I_{X/Y}$ is finite \'etale over $X$.
\end{lemma}
\begin{proof}
Since $X$ and $ X \times_Y {X}$ are proper \'etale over $Y$, the diagonal $\Delta:X\to X\times_Y X$ of $ X\to Y$ is proper \'etale \cite[Tag~0CIR]{stacks-project}. As the pull-back $X\times_{\Delta, X\times_Y X, \Delta} X$ of $\Delta$ along itself is the relative inertia stack $I_{X/Y}$ of $X\to Y$ (see \cite[Tag~034H]{stacks-project}),   the relative inertia stack $I_{X/Y}$ of $X\to Y$ is proper  \'etale over $X$. As $I_{X/Y}\to X$ is representable \cite[Tag~050P]{stacks-project} and proper \'etale, it is therefore  strongly representable  by Knutson's criterion \cite[Cor.~6.17]{Knutson}, hence finite \'etale as claimed.  
\end{proof}
	
	


\begin{lemma}\label{lem:degrees_of_BG_and_G} Let $S$ be an integral scheme. 
Let $G$ be a flat   finitely presented group scheme over $S$. Then $BG\to S$ is proper \'etale if and only if $G\to S$ is finite \'etale. In this case we have $\Ideg(BG\to S) = \deg(G \to S)$ and $$\deg(BG\to S) = \frac{1}{\deg(G\to S)}.$$
\end{lemma}
\begin{proof}   
Assume that $BG\to S$ is proper \'etale.
Then,  the inertia stack $I_{BG}\to BG$ is finite \'etale (Lemma \ref{lem:inertia}). Consider the canonical morphism $S\to BG = [S/G]$ corresponding to the trivial torsor $G\to S$. Let $\mathrm{Aut}^G(G)$ be the automorphism group of the trivial $G$-torsor $G\to S$ (as an object in the category of $G$-torsors over $S$). By definition, we have a Cartesian diagram 
\[\xymatrix{   \mathrm{Aut}^G(G) \ar[rr] \ar[d] & & I_{BG} \ar[d]^{\textrm{finite \'etale}} \\ S \ar[rr] & & BG.
}\]    Thus, the morphism $\mathrm{Aut}^G(G)\to S$ is finite \'etale as it is the pull-back of a finite \'etale morphism. However, as the automorphism group $\mathrm{Aut}^G(G)$ of the trivial (left) $G$-torsor is isomorphic to $G$  over $S$, the morphism  $G\to S$ is finite \'etale, as required.  Note that  the above also shows that $\Ideg(BG\to S) = \deg(G\to S)$. Since $\deg(S\to BG) =  \deg(S\times_{BG} S \to S) = \Ideg(BG\to S) = \deg(G\to S)$ and $\deg(S\to BG)\deg(BG\to S) = \deg(\mathrm{id}:S\to S)=1$ we obtain the last statement about the degree.

To conclude the proof, assume that  $G\to S$ is finite \'etale. Then, the morphism $S\to [S/G]=BG$ is finite \'etale. Therefore, as the composition $S\to BG\to S$ is the identity, the morphism $BG\to S$ is \'etale. Moreover, the above argument shows that the inertia stack of $BG\to S$ is finite, so that  the morphism $BG\to S$ is a coarse space map. In particular, it is proper by \cite[Thm.~6.12]{Rydh}.  This shows that $BG\to S$ is proper and \'etale.
\end{proof}


To study the degree and inertia degree in families, we will use the following rigidification  result for stacks (cf. \cite[Tag 04V2]{stacks-project}). Recall that we follow the conventions of the stacks project concerning gerbes (\cite[Tag~06QB]{stacks-project}).   
		
\begin{lemma} \label{lem:rigidification1}  
   	Let $f:X \to Y$ be a proper \'etale morphism of   algebraic stacks. Assume that $Y$ is a scheme.
	Then the morphism $f$ factorises as 
	\[
	\xymatrix{X \ar[r] \ar[dr]_{f}  & R \ar[d] \\ 
	 & Y}
	\]
	where $X \to R$ is a proper \'etale gerbe  and $R \to Y$ is a
	finite \'etale morphism of schemes.
\end{lemma}
\begin{proof} By Lemma \ref{lem:inertia}, the relative inertia stack $I_{X/Y}\to X$ is finite \'etale. 
By  standard rigidification results   \cite[Thm.~A.1]{AOV} (but see also for instance   \cite[Thm.~5.1.5]{ACV} or \cite[Thm.~5.1]{Romagny}), there exist an algebraic stack $R$ over $Y$ (denoted by $R=X\thickslash I_X$), a proper \'etale gerbe $X\to R$ (given by the ``rigidification'' $X\to X \thickslash I_X$), and a representable   morphism  $R\to Y$  such that the morphism in question factorises as $X \to R \to Y$. To conclude the proof, it suffices to show that $R\to Y$ is a finite \'etale morphism of schemes (note that a priori $R$ is just an algebraic space).


Firstly, note that  the morphism  $X\to R$  is surjective, as it is a gerbe.  Now, since  $X\to R$ is a  finitely presented \'etale surjective morphism and being \'etale is local on the source, it follows  that $R\to Y$ is \'etale.   
 
 Since  $X\to R\to Y$ is separated and $X\to R $ is surjective and universally closed,  a consideration of diagonals shows that $R\to Y$ is separated (cf. the proof of \cite[Tag~09MQ]{stacks-project}).

As $R \to Y$ is separated and \'etale, Knutson's criterion \cite[Cor.~II.6.17]{Knutson} implies that $R$ is a scheme. Now, as $R\to Y$ is a finitely presented separated morphism of schemes and $X\to Y$ is  a proper morphism with $X\to R$ surjective, it follows that  $R \to Y$ is proper \cite[Tag~0CQK]{stacks-project}. 
We see that $R \to Y$ is a proper \'etale morphism of schemes, hence finite \'etale. This concludes the proof.
\end{proof}

 We now show that the degree of a proper \'etale morphism is constant along the fibres over an integral base, as is familiar in the case of schemes.

\begin{lemma}	\label{lem:constant} 
	Let $X \to Y$ be a proper \'etale morphism of qcqs algebraic stacks 
	with $Y$ integral. Then for all points $y \in |Y|$ we have
	$$\deg(X_{y}/y) = \deg(X/Y).$$
\end{lemma}
\begin{proof} 
 	To prove the result we are allowed to take a dominant base-change. 
	Namely, let $Z$ be an integral stack with a dominant  map $Z \to Y$ whose
	image contains $y$. Then $X \times_Y Z \to Z$ is still proper \'etale,
	and it suffices to prove the equality of the degree for any point of $Z$
	above $y$. Therefore, we may assume that $Y$ 	is an integral scheme.
	Moreover, on passing to the normalisation, we may assume that $Y$ is normal. 
	In which case $X$ is also normal, and so its irreducible components  are
	its connected components. Therefore, the restriction of $X \to Y$ to an irreducible
	component of $X$ is again proper \'etale over $Y$. Hence summing over such components,
	we may assume that $X$ is integral.
	
	We first treat the case where $X \to Y$ is representable.
	Then $X \to Y$ is a proper \'etale algebraic space over a scheme, hence
	a scheme by Knutson's criterion \cite[Cor.~6.17]{Knutson}.
	So $X \to Y$ is a finite \'etale morphism of schemes;
	however this case is well-known and follows from the fact that
	the push-forward of $\OO_X$ to $Y$ is locally free of rank $\deg(X/Y)$.
	This proves the result for representable morphisms.
	
	We now treat the general case. We apply rigidification to $X \to Y$ (Lemma \ref{lem:rigidification1}) to see that $X\to Y$ factors as $X\to R\to Y$ with $X\to R$ a proper \'etale gerbe and $R\to Y$ a finite \'etale morphism of integral schemes.   As $R\to Y$ is representable (thus has inertia degree $1$), the stacks $I_{X/Y}$ and $I_{X/R}$ are isomorphic over $X$, so that we have   $\deg(I_{X/R}/X) = \deg(I_{X/Y}/X)$.
	Therefore,  by \eqref{eqn:degree_multiplicative} and Lemma \ref{lem:degrees_of_BG_and_G}  we have
	\begin{equation} \label{eqn:degs}
	\deg(X/Y) = \deg(X/R)\deg(R/Y) =  \frac{\deg(R/Y)}{\deg(I_{X/Y}/X)}.
	\end{equation}
  	The proper \'etale morphism $X_y\to y$ rigidifies as $X_y\to R_y\to y$. For $i \in I$ let $X_{y,i}$ and $R_{y,i}$ be the irreducible
  	components of $X_y$ and $R_y$, respectively. Then as in \eqref{eqn:degs}
  	we have
  	\begin{equation} \label{eqn:degsi}
  		\deg(X_y/y) = \sum_{i \in I} \deg(X_{y,i}/y)
  	= \sum_{i \in I} \frac{\deg(R_{y,i}/y)}{\deg(I_{X_{y,i}/y}/X_{y,i})}.
  	\end{equation}
   	Next by Lemma \ref{lem:inertia} the map $I_{X/Y} \to X$ is finite \'etale.
   	Thus applying the statement in this case to the inclusion of the generic
   	point of each $X_{y,i} \subset X$, we find that
   	$$\deg(I_{X_{y,i}/y}/X_{y,i}) = \deg(I_{X/Y}/X).$$
   	But $R \to Y$ is also finite \'etale, so 
   	we have $\deg(R/Y) = \deg(R_y/y) = \sum_{i \in I} \deg(R_{y,i}/y)$.
   	Combining these with \eqref{eqn:degs} and \eqref{eqn:degsi}
   	completes the proof.
\end{proof}

 \section{Hermite--Minkowski for algebraic stacks} \label{sec:HM}
In this section we prove our version of  Hermite--Minkowski's theorem for algebraic stacks (Theorem \ref{thm:HM_intro}). We begin with some results on gerbes.

\subsection{Classifying proper \'etale gerbes}\label{section:classifying_gerbes}
We first explain the classification of proper \'etale gerbes over a scheme $S$ using bands, following  Giraud \cite{Gir71}. 

A band over a scheme $S$ is an $S$-object of the stack of bands, as defined in 
\cite[D\'ef.~IV.1.1.6]{Gir71}. (\emph{Band} is the current accepted English translation of the french \emph{lien}).   If $G$ is a group scheme over $S$, then we let $\lien(G)$ be the associated band over $S$. One can associate to every gerbe over $S$ a band over $S$ \cite[\S IV.2.2]{Gir71}. 

In general the band of a gerbe over $S$ is not necessarily of the form $\lien(G)$ for some $G$ (see Remark \ref{remark:kresch}).
However, for a \emph{proper \'etale gerbe}, we have the following weaker statement.

 \begin{lemma}\label{lem:degree_of_gerbe} Let $S$ be an integral scheme
and $X\to S$  a proper \'etale gerbe. There is a     finite (abstract) group $G$ with $\# G = \mathrm{Ideg}(X/S) $  such that, locally for the \'etale topology on $S$, the band of $X\to S$ is   isomorphic to $\lien(G_S)$.
 \end{lemma}
 \begin{proof}
Let $G$ be the finite (abstract) group given by the inertia group $I_{\bar{\eta}}$, where $\bar{\eta}$ denotes a geometric point over the generic point $\eta$ of $X$.
 Since $X\to S$ is proper \'etale, it follows  that there is a   scheme $U$, a finitely presented  \'etale surjective morphism $U\to S$, 
and a  flat finitely presented group scheme  $\mathcal{G}$ over $U$ such that $X\times_S U \cong B\mathcal{G}$. 
Note that, as $X\to S$ is proper \'etale, the morphism $B\mathcal{G}\to U$ is proper \'etale. Thus, by  Lemma \ref{lem:degrees_of_BG_and_G},  the group scheme $\mathcal{G}\to U$ is finite \'etale. We now trivialize the group scheme $\mathcal{G}$ over $U$. 
Let $V\to U$ be a finitely presented surjective \'etale   morphism  such that $\mathcal{G}_V$ is constant over each connected component of $V$. Since the fibre over every maximal point of $U$ is $G$, we have $\mathcal{G}_V \cong G_V$. Moreover $X_V \cong B\mathcal{G}_V$, thus $X_V \cong BG_V$. It follows that $\lien(X\to S)$ is isomorphic to $\lien(BG_V) = \lien(G_S)_V$ over $V$, as required.  
 \end{proof}

 \begin{remark}\label{remark:kresch} If $X\to S$ is a proper \'etale \emph{abelian} gerbe, then there is an abelian group $G$ such that the band of $X$ is isomorphic to $\lien(G)$ over $S$. However, this is not necessarily true if $X\to S$ is not abelian, as the following example shows. (This example was communicated to us by Andrew Kresch).
 
Let $ \ZZ/4\ZZ$ act as the automorphism group of $\ZZ/5\ZZ$ and let $ G = \ZZ/5\ZZ \rtimes \ZZ/4\ZZ$ be the corresponding
semidirect product. There is an exact sequence of groups
$$
1\to D_5\to G\to \ZZ/2\ZZ\to 1,
$$ where $D_5$ is the  (non-abelian) dihedral group of order 10.
Let $K$ be an imaginary quadratic extension of $\QQ$ and consider $ X:=[\Spec(K)/G]$,
where $G$ acts through $\ZZ/2\ZZ=\mathrm{Gal}(K/\QQ)$. Note that $ X$  is the gerbe of ``lifts of the $\ZZ/2\ZZ$-torsor $K/\QQ$ to a $G$-torsor''. This is a gerbe over $S:=\Spec \QQ$, and it is \'etale locally
isomorphic to   $B(D_5)$. The outer automorphism group of $D_5$ is
$\ZZ/2\ZZ$, so the band splits over  a quadratic extension of $\QQ$, namely $K$.  Suppose that there is a finite \'etale group scheme $H$ over $\QQ$ such that the band of $ X$ is isomorphic to $\lien(H)$. Then $H$ is a form of $D_5$ (corresponding to an element in $\mathrm{H}^1(\QQ, \Aut(D_5)) = \mathrm{H}^1(\QQ,G)$), and   there would have to be a continuous
homomorphism $\mathrm{Gal}(\QQ)\to\Aut(D_5)=G$ such that the composite $\mathrm{Gal}(\QQ)\to G \to \ZZ/2\ZZ$  
corresponds to the quadratic extension $K$. In particular, this would induce  a homomorphism  
$\mathrm{Gal}(\QQ)\to \ZZ/4\ZZ\to \ZZ/2\ZZ$, contradicting the fact that, for any cyclic degree $4$ extension of $\QQ$, the intermediate
field is always real quadratic (the discriminant of the quadratic extension being a sum of two squares).  Thus the  band of the proper \'etale gerbe $X$ is not isomorphic to the band induced by a group scheme over $\QQ$.
 \end{remark}

We next recall how to classify bands over $S$ which are locally isomorphic to $\lien(G)$.

\begin{lemma}\label{lem:class_of_band} Let $S$ be a scheme. Let $G$ be a flat finitely presented group scheme over $S$. Then the set of isomorphism classes of bands over $S$  which are fppf locally isomorphic to the band $\lien(G)$ (in the category of bands) is in bijection with $\HH^1(S, \Out(G))$.
\end{lemma}
\begin{proof}
See \cite[Cor.~IV.1.1.7.3]{Gir71}.
\end{proof}

If $L$ is a band over a scheme $S$, we let  $Z(L)$ be the center of $L$ \cite[\S IV.1.5.3.2]{Gir71}. Note  that $Z(L)$  is an \emph{abelian} band over $S$ (as defined in \cite[Prop.~IV.1.2.3]{Gir71}), and  therefore ``is'' a commutative group scheme over $S$ (see again \cite[\S IV.1.5.3.2]{Gir71}). If $G$ is a finite group, then the center of $\lien(G)$ is given by $\lien(Z(G))$, and the latter can be identified naturally with $Z(G)$.

If $L$ is a band over a scheme $S$, then we follow Giraud and let $\HH^2(S,L)$ be the set of $L$-equivalence classes of $S$-gerbes banded by $L$; see \cite[Def.~IV.3.1.1]{Gir71}. One relates this to the  second cohomology group of $Z(L)$ via the following.

\begin{lemma}\label{lem:class_of_banded_gerbe} Let $S$ be an integral  scheme.
Let $X\to S$ be a proper \'etale gerbe. Let $L$ be its band over $S$.  Then the center $Z:=Z(L)$ of the band $L$ is a finite \'etale group scheme over $S$ of degree dividing $\mathrm{Ideg}(X/S)$, and the set  of $L$-equivalence classes of proper \'etale gerbes  over $S$ banded by $L$ is in bijection with $  \HH^2(S,Z)$.  
\end{lemma}
\begin{proof}  By Lemma \ref{lem:degree_of_gerbe}, there is a finite group $G$ of order $\mathrm{Ideg}(X/S)$ such that $L$ is isomorphic to $\lien(G)$, locally for the fppf topology on $S$. In particular $Z(L)$ is locally isomorphic to $\lien(Z(G))$, so that $Z(L)$ is a finite \'etale group scheme over $S$ of degree $\#Z(G)$. Since $\#Z(G)$ divides $\# G = \mathrm{Ideg}(X/S)$ this shows   the first statement.

 To prove the second statement,  note that every gerbe over $S$ banded  by $L$ is proper \'etale over $S$. (Indeed, by descent it suffices to show this holds for a neutral  gerbe $Y \to S$ banded by $L$. However, the band of $Y \to S$ is isomorphic to $\lien(G)$ over $S$, so that $Y\cong BG$. Therefore, $Y\to S$ is proper \'etale by Lemma \ref{lem:degrees_of_BG_and_G}.) Therefore $\HH^2(S,L)$ is the set of $L$-equivalence classes of proper \'etale gerbes over $S$. But there is an action of $\HH^2(S,Z)$  on $\HH^2(S,L)$ which is free and transitive    \cite[Thm.~IV.3.3.3.(i)]{Gir71},
 and as  $\HH^2(S,L)$ is non-empty (it contains the class of $X \to S$), the result follows.
\end{proof}

\subsection{Hermite--Minkowski for gerbes} We now prove a version of Theorem \ref{thm:HM_intro} for gerbes.  We require the following finiteness statements for cohomology sets.

\begin{lemma} \label{lem:H^1_finite} 
	Let $B$ be a connected
	arithmetic scheme and $G$ a finite \'etale  group scheme over $B$. The following statements hold.
	\begin{enumerate}
	\item The pointed set $\HH^1(B, G)$ is finite.
	\end{enumerate}
	Assume further that $G$ is abelian.
	\begin{enumerate}
	\item[(2)] If $\dim B =1$, then the set $\HH^r(B,G) $ is finite for all $r\geq 0$.
	\item[(3)]  If  the degree of $G$ over $B$ is invertible on  $B$,
	then the set $\HH^r(B, G)$ is   finite for all $r \geq 0$.
	\end{enumerate}
\end{lemma}
\begin{proof}  
	As $G$ is finite \'etale, by Hermite--Minkowski for arithmetic schemes \cite{smallness} there are only finitely many possibilities
	for the scheme underlying the torsor representing each element of $\HH^1(B,G)$.
	As there are also only finitely many possibilities for the $G$-action on such a scheme,
	this proves the finiteness of $\HH^1(B,G)$. 
	The second statement follows from  	\cite[Thm.~II.3.1]{Mil06}. 
	The third statement follows from \cite[Prop.~II.7.1]{Mil06}. 
\end{proof} 

\begin{remark}  We do not know whether the cohomology groups $\HH^r(B,G)$ are always finite if $\dim B > 1$ and the degree of the finite \'etale group scheme $G$  is not invertible on $B$.
%
%
\end{remark}

Our main finiteness result is for gerbes of fixed inertia degree over arithmetic schemes, and reads as follows.
\begin{theorem}\label{thm:HM_for_gerbes}
	Let $B$ be an integral arithmetic scheme and $e$ a positive integer. If $\dim B > 1$, then assume that  $e$ is invertible on $B$. Then the set of $B$-isomorphism classes of proper \'etale gerbes $X\to B$ over $B$ of inertia degree $e$ is finite.
\end{theorem}
 
\begin{proof}   We prove the result using the classification results from the previous section.
	
	Let $X\to B$ be a proper \'etale gerbe   of inertia degree $e$, and note that $X$ is integral.   By Lemma \ref{lem:degree_of_gerbe}, there is a finite (abstract) group $G$ of order $e$ such that the band of $X\to B$ is locally isomorphic to $\lien(G)$ in the category of bands. By Lemma \ref{lem:class_of_band}, the band of $X\to B$ is canonically an object of $\mathrm{H}^1(B, \mathrm{Out}(G))$. Note that, as $\mathrm{Out}(G)$ is a finite \'etale group scheme over $B$, the latter set is finite (Lemma \ref{lem:H^1_finite}).   Thus, as there are only finitely many groups of order $e$, there are finitely many bands $L_1, \ldots, L_r$ over $B$ such that any proper \'etale gerbe of inertia degree $e$ is banded by some $L_i$. 
	
	Fix $i\in \{1,\ldots, r\}$ and define $L:= L_i$. Note that the center  $Z(L)$ of $L$ is a finite \'etale group scheme over $B$ and that the set of $L$-equivalence classes of proper \'etale gerbes over $B$ banded by $L$ is in bijection with $\HH^2(B,Z(L))$; see Lemma \ref{lem:class_of_banded_gerbe}. Since $Z(L)$ is abelian of order dividing $e$, the  group $\mathrm{H}^2(B,Z(L))$ is  finite under our assumptions (Lemma \ref{lem:H^1_finite}).
	
	From the above we conclude that the set of pairs $(L,X)$ with $L$ a   band over $B$ and   $X$ a (proper \'etale) gerbe banded by $L$  of inertia degree $e$ over $B$   is finite, up to equivalence as a pair.   However, the latter (finite) set surjects onto the set of $B$-isomorphism classes of proper \'etale gerbes over $B$ of inertia degree $e$, which is therefore finite, as required.
\end{proof}

\begin{remark}
Let $G$ be a finite abelian group, let $S =\Spec \QQ$, and let $X\to S$ be a non-neutral (proper \'etale) gerbe banded by $G$, corresponding to some   element $[X]$ in $\HH^2(S,G)$ with $2[X]\neq 0$.  Let $X^{-}\to S$ be the gerbe banded by $G$ corresponding to the element $-[X]$ in $\HH^2(S,G)$; see \cite[Prop~IV.3.3.2.(iii)]{Gir71}. Note that, as $2[X]\neq 0$, the stacks $X$ and $X^{-}$ are not isomorphic as gerbes banded by $G$. However, by the definition of $X^{-}$ (and \cite[Prop.~IV.2.1.7.2]{Gir71}), it follows that $X$ and $X^{-}$ are isomorphic as algebraic stacks (hence  gerbes) over $S$.
\end{remark}

\subsection{Hermite--Minkowski for proper \'etale morphisms}
The following result generalizes Theorem \ref{thm:HM_for_gerbes} from proper \'etale gerbes to proper \'etale morphisms. 

\begin{theorem}\label{Cor:HM}  
Let $B$ be an integral arithmetic scheme and $n,m \in \NN$. If $\dim B > 1$, assume that $m!$ is invertible on $B$. Then, the set of $B$-isomorphism classes of proper \'etale morphisms $X\to B$   of degree at most $n$ such that the inertia degree of every irreducible component of $X$ is at most $m$, is finite. 
\end{theorem}
\begin{proof}  
 By rigidification (Lemma \ref{lem:rigidification1}), for any proper \'etale morphism $X\to B$ as in the statement, there is a scheme $R$ such that the morphism $X\to B$ factors  as $X \to R \to B$, where $X\to R$ is a proper \'etale gerbe and $R\to B$ is finite \'etale. 
As $R\to B$ is finite \'etale, it follows that $R$ is an arithmetic scheme. 

Let $X_i$ for $i \in I$ be the irreducible components of $X$, and $R_i$ the corresponding components of $R$.
Now, by \eqref{eqn:degree_multiplicative} and the relation between the degree and inertia degree of a proper \'etale gerbe (Lemma \ref{lem:degrees_of_BG_and_G}), we have
$$
	n \geq \deg(X/B) = \sum_{i \in I} \deg(X_i/B) =
	\sum_{i \in I} \deg(X_i/R_i)\deg(R_i/B) = 
	\sum_{i \in I} \frac{\deg(R_i/B)}{\Ideg(X_i/R_i)}.
$$
However, as $R_i\to B$ is representable, we have $ \Ideg(X_i/R_i) = \Ideg(X_i/B)$. As $ \Ideg(X_i/R_i) = \Ideg(X_i/B)  \leq m $ we find that
$$
	nm \geq \sum_{i \in I} \deg(R_i/B) = \deg(R/B).
$$
Since $\deg(R/B)$ is bounded, the set of $B$-isomorphism classes of the schemes $R$ occurring above is finite by the classical Hermite--Minkowski theorem (see \cite[p.~209]{FaltingsComplements} or \cite{smallness}). 

Next $\Ideg(X_i/R_i) \leq m$, hence $\Ideg(X_i/R_i)$ is bounded and invertible on $R_i$, so the set of $R_i$-isomorphism classes  of proper \'etale gerbes  $X_i\to R_i$ occurring above is finite by our version of Hermite--Minkowski for proper \'etale gerbes (Theorem \ref{thm:HM_for_gerbes}). This implies that the set of $B$-isomorphism classes of proper \'etale gerbes $X\to R$ occurring above is also finite. This completes the proof.
\end{proof}

We now prove the  version formulated in the introduction.

\begin{proof}[Proof of Theorem \ref{thm:HM_intro}] 
As $X$ is integral and the statement is up to $K$-isomorphism (as opposed to $A$-isomorphism),  replacing $\Spec A$ by an affine open if necessary  we may assume that  $n!$ is invertible in $A$. The result now follows from Theorem \ref{Cor:HM}. 
\end{proof}

\section{Arithmetic hyperbolicity}\label{section:arithmetic_hyperbolicity}
In this section we extend the notion of arithmetic hyperbolicity    (employed for instance in \cite[\S 2]{UllmoShimura}, and \cite{Autissier1, Autissier2}) to algebraic stacks, and establish its formal properties.  
 Throughout this section $k$ is an algebraically closed field of characteristic zero.

\subsection{Definition of arithmetic hyperbolicity}

 If $X$ is a variety over a number field $K\subset\Qbar$, the finiteness of the set of $\OO_{L}[S^{-1}]$-integral points (on an appropriate choice of model) for all number fields $K \subset L$  and all finite sets of finite places $S$ of $L$ depends only on the isomorphism class of $X_{\Qbar}$.

Our aim is to extend this natural property of a variety over $\Qbar$ to algebraic stacks defined over our algebraically closed field $k$ (which might not be $\Qbar$).

Let $A\subset k$ be a $\ZZ$-finitely generated subring and let $\mathcal{X}$ be a finitely presented algebraic stack over $A$.
As the $A$-valued points $\mathcal{X}(A)$ on  $\mathcal{X}$ form a groupoid (which is not naturally a set in general), to study the finiteness of integral points, we consider    the set $\pi_0(\mathcal{X}(A))$ of isomorphism classes of objects in $\mathcal{X}(A)$. Moreover, the most flexible notion of arithmetic hyperbolicity is obtained by considering the image of $\pi_0(\mathcal{X}(A))$ in the set of ``geometric'' points $\pi_0(\mathcal{X}(k))$. This ``flexibility'' will help in establishing some of the basic geometric properties we require. For numerous applications it is also useful to consider the image of the integral points inside the rational points (cf.~the map \eqref{Faltings}).

\begin{definition}\label{defn:arithmetic_hyperbolicity}[Arithmetic hyperbolicity]
	A finitely presented algebraic stack  $X$ over an algebraically closed field $k$ of characteristic $0$ is called \emph{arithmetically hyperbolic (over $k$)} if there exist a $\ZZ$-finitely generated subring    $A\subset k$ and a model $\mathcal X$ of $X$ over $A$  such that, for all $\ZZ$-finitely generated subrings  $A'\subset k$ containing $A$, the set 
	$$\Im[ \pi_0(\mathcal X(A'))~\to~\pi_0(\mathcal X(k))]$$  is finite. 
 \end{definition}

 \begin{remark} 
 If $X$ is a variety over $k$, then $X$ is arithmetically hyperbolic over $k$ if and only if there is a $\ZZ$-finitely generated subring $A\subset k$ and a  model $\mathcal{X}$ of $X$ over $A$ such that $\mathcal{X}$ is a separated scheme and, for all $\ZZ$-finitely generated subrings $A'\subset k$ containing $A$, the set $\mathcal{X}(A')$ is finite. (This follows from Lemma \ref{lem:arithm_hyp} and the fact that $\pi_0(\mathcal{X}(A)) = \mathcal{X}(A)$ injects into $\pi_0(X(k)) = X(k)$.)
 
 Consider the important case $k = \bar{\QQ}$. Then a $\ZZ$-finitely generated subring $A \subset \bar\QQ$ is an order in some number field. However in the study of arithmetic hyperbolicity, we are free to replace $\Spec A$ by a dense open subset, as it only makes the problem more difficult, so we may assume that $\Spec A$ is actually regular. We deduce that $X$ is arithmetically hyperbolic if and only if there is a number field $K$ and some model $\mathcal{X}$ for $X$ over $\OO_K$, such that $\mathcal{X}(\OO_L[S^{-1}])$ is finite for all number fields $K \subset L$ and all finite sets of places $S$ of $L$. A more general version of this statement is provided by Lemma \ref{lem:smooth}.
 \end{remark}

 \begin{remark}  
 The fact that we work with stacks leads to some pathologies that are worth keeping in mind. For instance, a rational point on an algebraic stack can come from infinitely many pairwise non-isomorphic integral points. Indeed, there is a $\ZZ$-finitely generated subring $A\subset \CC$ such that the map $\mathrm{Pic}(A) = \pi_0(B\mathbb{G}_m(A)) \to \pi_0(B\mathbb{G}_{m}(\mathrm{Frac}{A})) = \mathrm{Pic}(\mathrm{Frac}(A)) = \{1\}$  has infinite fibres. The problem here is that the stack $B\Gm$ is non-separated.
 
Another phenomenon is that infinitely many rational points can give rise to the same geometric point. Namely, consider $B\mu_2$ over $\QQ$. We have $\pi_0(B\mu_2(\QQ)) = \QQ^*/\QQ^{*2}$, yet $\pi_0(B\mu_2(\bar{\QQ}))$ is a singleton. It is for these reasons that we consider the image of the integral points inside the geometric points in Definition \ref{defn:arithmetic_hyperbolicity}.  

However, for finitely presented separated Deligne--Mumford stacks, being arithmetically hyperbolic is equivalent to the more natural and a priori stronger condition of having only finitely many (isomorphism classes of) integral points; see  Theorem~\ref{thm:equivalences} for a precise statement.
 \end{remark}

 \begin{remark} Note that $\mathbb{Z}$ acts on $\mathbb{A}^1_k$ via translation. Since the action is free, the algebraic stack $X=[\mathbb{A}^1/\mathbb{Z}]$ is an algebraic space. Note that $X$ is of finite type over $\mathbb{C}$. However, by \cite[Tag~06Q2]{stacks-project}, the stack $X$ is not finitely presented over $\CC$, as its diagonal is not quasi-compact. However, note that a finite type algebraic stack over $\CC$ with affine diagonal (or quasi-compact diagonal) is in fact finitely presented (by definition) over $\CC$. Thus, the distinction between finite type and finitely presented only appears when the stack in question is ``highly'' non-separated. \end{remark}
 
\begin{example}[Faltings]\label{example:elliptic_curve} In \cite{Faltings2} Faltings proved Mordell's conjecture which says that a smooth proper geometrically connected curve $C$ of genus at least two over  a number field $K$ has only finitely many $K$-rational points. In \cite{FaltingsComplements} he also proved the more general statement that a smooth proper connected curve of genus $g$ is arithmetically hyperbolic over $k$ if and only if $g \geq 2$. 
\end{example}

 \begin{example}[Faltings] \label{ex:Shaf_Conj} Let $g\geq 2$ be an integer.
Let $\mathcal{M}_g$ be the finite type separated Deligne--Mumford   stack of smooth proper curves of genus $g$ over $\ZZ$. Then $\mathcal{M}_{g,k}$ is arithmetically hyperbolic over $k$.  This is a reformulation of Faltings's celebrated finiteness theorem (\emph{formerly} Shafarevich's conjecture) for genus $g$ curves over a number field $K$ with good reduction outside a given finite set of finite places of $K$.
\end{example}

\begin{remark}[Faltings]
Let $A$ be an abelian variety over $\mathbb{C}$, and let $X\subset A$ be a closed subvariety. Then $X$ is arithmetically hyperbolic over $\mathbb{C}$ if and only if $X$   does not contain a translate of a positive-dimensional abelian subvariety of $A$. This is a consequence of Faltings's theorem  \cite{FaltingsLang}. Generalizations to subvarieties of semi-abelian varieties were obtained by Vojta \cite{Vojta1, Vojta2}.
Non-trivial affine examples of arithmetically hyperbolic varieties are given in \cite{Autissier1, Autissier2, CLZ, CorZanAnnals, FaltingsLang, Levin, VojtaSub}. 
\end{remark}

 \subsection{Basic properties of arithmetically hyperbolic stacks} 
 
\begin{lemma}[Independence of model]\label{lem:arithm_hyp}
	Let $X$ be a finitely presented arithmetically hyperbolic algebraic stack over $k$. Then, for all $\ZZ$-finitely generated subrings $B\subset k$ and all models $\mathcal Y$ for $X$ over $B$, the set $\Im[\pi_0(\mathcal Y(B))\to \pi_0(\mathcal Y(k))]$ is finite.
\end{lemma}
\begin{proof}
Since $X$ is arithmetically hyperbolic, there exist a $\ZZ$-finitely generated subring    $A\subset k$ and a model $\mathcal X$ of $X$ over $A$  such that, for all $\ZZ$-finitely generated subrings  $A'\subset k$ containing $A$, the set 
$\Im[ \pi_0(\mathcal X(A'))~\to~\pi_0(\mathcal X(k))]$  is finite. Now, let $B\subset k$ and $\mathcal Y$ be as in the statement of the lemma. Note that there is a $\mathbb Z$-finitely generated subring $C\subset k$ containing     $A$ and $B$ such that $\mathcal X_C \cong \mathcal Y_C$. It follows that $\Im[\pi_0(\mathcal Y(B))\to \pi_0(\mathcal Y(k))]$ is a subset of 
\[\Im[\pi_0(\mathcal Y(C))\to \pi_0(\mathcal Y(k))] = \Im[\pi_0(\mathcal X(C))\to \pi_0(\mathcal X(k))].\] As the latter set is finite, this concludes the proof.
\end{proof}
 
A homomorphism of commutative rings $A \to B$ is called \emph{smooth} if the morphism $\Spec B \to \Spec A$ is smooth.

 \begin{lemma}[Can test on smooth subrings]\label{lem:smooth}   A finitely presented  algebraic stack $X$ over $k$ is arithmetically hyperbolic over $k$ if and only if there exists a
  $\ZZ$-finitely generated subring    $A\subset k$ and a model $\mathcal X$ of $X$ over $A$  such that, for all $\ZZ$-finitely generated  subrings  $A \subset A'\subset k$ which are \textbf{smooth} over $A$, the set 
	$$\Im[ \pi_0(\mathcal X(A'))~\to~\pi_0(\mathcal X(k))]$$  is finite. 
 \end{lemma}
 
\begin{proof} 

The first implication is clear. For the reverse implication, let $  B\subset k$ be a $\ZZ$-finitely generated subring containing $A$. Since the finitely presented morphism $\Spec B\to \Spec A$ is generically smooth, one can find a $\ZZ$-finitely generated subring   $ B \subset A'\subset k$ which is smooth over $A$. Thus by assumption, the set  $\Im[\pi_0(\mathcal X(A'))\to \pi_0(\mathcal X(k))]$ is finite. This implies that the subset 
 $$\Im[\pi_0(\mathcal X(B))\to \pi_0(\mathcal X(k))]\subset  \Im[\pi_0(\mathcal X(A'))\to \pi_0(\mathcal X(k))]$$ is finite, and concludes the proof.
\end{proof}

%
%

  \begin{remark}\label{REMARK:lim_arg}
If $f:X\to Y$ is a (finitely presented) morphism of finitely presented algebraic stacks over $k$, then there is a $\ZZ$-finitely generated subalgebra $A\subset k$, and (finitely presented) morphism $F:\mathcal{X}\to \mathcal{Y}$ of finitely presented algebraic stacks over $A$ such that  $F_k \cong f$.
 \end{remark}
 
 The following simple lemma can be viewed as confirming the intuitive statement   that a ``space which fibers in hyperbolic spaces over a hyperbolic variety'' is hyperbolic. 
 
  \begin{lemma}[Fibration property]\label{lem:formal_prop}
 	Let $  Y\to Z$ be a    morphism of finitely presented algebraic stacks over $k$. If $Z$ is arithmetically hyperbolic and, for all geometric points $z:\Spec k\to Z$ of $Z$, the algebraic stack $Y_z$ is arithmetically hyperbolic, then 
 	$Y$ is arithmetically hyperbolic. 
 \end{lemma}  
 \begin{proof} We first use Remark \ref{REMARK:lim_arg} and the arithmetic hyperbolicity of $Z$ over $k$ to spread out $Y\to Z$ over some $A\subset k$. More precisely, 
let $A\subset k$ be an integrally closed $\mathbb Z$-finitely generated subring and let $\mathcal Y\to \mathcal Z$ be a morphism of finitely presented algebraic stacks over $A$ which is isomorphic to $Y\to Z$ over $k$ such that, for all $\ZZ$-finitely generated subrings  $A'\subset k$ containing $A$, the set   
	$$\Im[ \pi_0(\mathcal Z(A'))~\to~\pi_0(\mathcal Z(k))]$$  is finite.     Let $A'\subset k$ be a $\mathbb Z$-finitely generated subring.  To prove the lemma, it suffices to show that the set 
 	\[ \Im[ \pi_0 (\mathcal Y(A')) \to \pi_0(\mathcal Y(k))]\] is finite. To do so, consider the natural morphism of sets 
 	\[
 	  \Im[ \pi_0 (\mathcal Y(A')) \to \pi_0(\mathcal Y(k))]\to \Im[ \pi_0(\mathcal Z(A'))~\to~\pi_0(\mathcal Z(k))]
 	\]  Let $z:\Spec k\to Z$ be a point in the image of the latter morphism of sets, and let $\widetilde{z}:\Spec A'\to \mathcal{Z}$ be an extension of $z$ over $A'$. Since $\Im[ \pi_0(\mathcal Z(A'))~\to~\pi_0(\mathcal Z(k))]$ is finite, it suffices to show that the fibre over $z$ is finite.  However, the fibre over $z$ is contained in the set $$  \Im[ \pi_0 (\mathcal Y_{\widetilde{z}}(A')) \to \pi_0(\mathcal Y_z(k))] .$$ Since $Y_z$ is arithmetically hyperbolic and $\mathcal{Y}_{\widetilde{z}}$ is a model for $Y_z$, the latter set is finite (Lemma \ref{lem:arithm_hyp}).
 \end{proof}

\begin{lemma}\label{lem:c} Let $\sigma:k\to L$ be a morphism of algebraically closed fields of characteristic zero. Let $X$ be a finitely presented algebraic stack over $k$. If $X_L = X\otimes_{k,\sigma} L$ is arithmetically hyperbolic over $L$, then $X$ is arithmetically hyperbolic over $k$.
\end{lemma}
\begin{proof}
If $A\subset k$ is a  $\ZZ$-finitely generated subalgebra $A\subset k$, then $\sigma(A)\subset L$ is  a $\ZZ$-finitely generated subalgebra of $L$. This observation easily allows one to conclude.
\end{proof}

\begin{remark}[Persistence Conjecture]
It seems reasonable to suspect that the converse of Lemma \ref{lem:c} holds. The so-called ``Persistence Conjecture'' predicts this for varieties (see \cite[Conjecture~1.15]{vBJK}), but it might be just as reasonable for stacks. More precisely, let $L/k$ be an extension of algebraically closed fields of characteristic zero and let $X$ be a variety over $k$. Then,  the  Persistence Conjecture says that  the variety $X$ is arithmetically hyperbolic over $k$ if and only if $X_L$ is arithmetically hyperbolic over $L$. This conjecture is verified for projective   varieties under additional ``boundedness'' or ``hyperbolicity'' conditions in \cite[\S 4]{Jaut}, e.g., if $X$ is algebraically hyperbolic over $k$. Moreover, it is shown to hold for varieties with a quasi-finite (complex-analytic) period map in \cite[Theorem~1.4]{JLitt}, hyperbolically embeddable smooth affine  varieties in \cite[Theorem~1.4]{JLevin}, varieties with a quasi-finite morphism to a semi-abelian variety in \cite[Theorem~7.4]{vBJK}, and certain moduli spaces  of polarized varieties in \cite[Theorem~1.5]{JSZ}. 
\end{remark}

If $\sigma$ is an element of $\Aut(k)$ and $X$ is an  algebraic stack over $k$, we let $X^\sigma$ be the algebraic stack $X\times_{k,\sigma} k$.
\begin{lemma}[Conjugation property]\label{LEMMA:conj}
Let $\sigma$ be a field automorphism of $k$. Let $X$ be a finitely presented algebraic stack over $k$. If $X$ is arithmetically hyperbolic over $k$, then $X^\sigma$ is arithmetically hyperbolic over $k$.
\end{lemma}
\begin{proof} This follows from Lemma \ref{lem:c}.
\end{proof}


 \begin{lemma}\label{lem:G_gerbes_are_AH}
 Let $X\to Y$ be a morphism of finitely presented algebraic stacks over $k$. If $Y$ is arithmetically hyperbolic and $X\to Y$ is a gerbe, then $X$ is arithmetically hyperbolic.
 \end{lemma}
 \begin{proof} If $B$ is a gerbe over $k$, then  $\pi_0(B(k))$ is a singleton.   Thus, a gerbe over $k$ is arithmetically hyperbolic over $k$. Therefore, for all $y$ in $Y(k)$, as the fibre $X_y$  is a gerbe over $k$, we see that the fibres of $X\to Y$ are  arithmetically
	hyperbolic. Therefore,    the lemma follows from the fibration property (Lemma \ref{lem:formal_prop}).
 \end{proof}

	

\begin{lemma}\label{lem:reducedness}
	Let $X$ be a finitely presented algebraic stack over $k$, and let $X_{\textrm{red}}$ be the associated reduced algebraic stack. The algebraic stack $X$ is arithmetically hyperbolic if and only if $X_{\textrm{red}}$ is arithmetically hyperbolic.
\end{lemma}
\begin{proof} 
	Note that the fibres of the morphism $X_{\textrm{red}}\to X$ are arithmetically hyperbolic. In particular, if $X$ is arithmetically hyperbolic, then   $X_{\textrm{red}}$ is arithmetically hyperbolic (Lemma \ref{lem:formal_prop}). Conversely, 
	assume that $X_{\textrm{red}}$ is arithmetically hyperbolic. 
	Let $A\subset k$ be a finitely generated $\ZZ$-algebra and let $\mathcal{X}$ be a model for $X$ over $A$.  Since $A$ is an integral domain,  every morphism $\Spec A\to \mathcal X$ factors uniquely via $\mathcal{X}_{\textrm{red}}$. Thus, the natural map of sets
	\[ \Im[ \pi_0(\mathcal X_{\textrm{red}}(A))~\to~\pi_0(\mathcal X_{\textrm{red}}(k))] \to \Im[ \pi_0(\mathcal X(A))~\to~\pi_0(\mathcal X(k))] \] is surjective.  Since $X_{\textrm{red}}$ is arithmetically hyperbolic, the set $$\Im[ \pi_0(\mathcal X_{\textrm{red}}(A))~\to~\pi_0(\mathcal X_{\textrm{red}}(k))] $$ is finite (Lemma \ref{lem:arithm_hyp}). We conclude that $X$ is arithmetically hyperbolic.
\end{proof}

 \begin{proposition}\label{prop:qf}
 	Let $  Y\to Z$ be a quasi-finite morphism of finitely presented algebraic stacks over $k$. If $Z$ is arithmetically hyperbolic, then 
 	$Y$ is arithmetically hyperbolic. 
 \end{proposition}  
 \begin{proof}
 	By Lemma  \ref{lem:formal_prop}, it suffices to show that the    $k$-fibres of $Y\to Z$ are arithmetically hyperbolic. However, as the   $k$-fibres are     quasi-finite (finitely presented) algebraic stacks over $k$, they are clearly arithmetically hyperbolic.
 \end{proof}

\begin{lemma}\label{lem:irreducible}
		Let $X$ be a finitely presented algebraic stack over $k$. Then $X$ is arithmetically hyperbolic if and only if all irreducible components of $X$  are arithmetically hyperbolic.
\end{lemma}
\begin{proof}
Suppose that $X$ is arithmetically hyperbolic. Let $Z$ be an irreducible component of $X$. Since $Z\to X$ is quasi-finite, it follows that $Z$ is arithmetically hyperbolic (Proposition \ref{prop:qf}). The converse follows from the fact that $X$ has only finitely many irreducible components.
\end{proof}

 \subsection{The twisting lemma} The notion of arithmetic hyperbolicity is a priori a condition on $k$-points, with $k$ algebraically closed. We show in this section that under certain assumptions on the model, we can in fact deduce finiteness results for $A$-valued points.
 We refer to this result as the ``twisting lemma (for arithmetic hyperbolicity)'', as it is concerned with the finiteness of twists of a given $A$-object of a stack.

 The twisting lemma below  generalizes our previous  results for canonically polarized varieties \cite[Lem.~4.1]{Jav15}, complete intersections \cite[Thm.~4.10]{JL}, and certain Fano varieties \cite[Prop.~4.7]{JLFano}. It is very useful in applications; it says that, in certain cases, to deduce finiteness of (isomorphism classes of) integral points, it suffices to prove the a priori weaker claim that $X$ is arithmetically hyperbolic. This latter property  can be tackled using more geometric techniques.  We give such applications in \S\ref{sec:appplications}.

The following is the key step, which says that under suitable assumptions, there are only finitely many isomorphism classes of integral points $y$ that become isomorphic to a given integral point $x$ over the algebraic closure.  The ``finite diagonal'' condition in the statement holds for example for  finite type separated Deligne--Mumford stacks. (Indeed, as unramified morphisms are locally quasi-finite \cite[Tag~02V5]{stacks-project}, the diagonal of such a stack is proper and unramified, hence   finite unramified.)  Recall that we work over an algebraically closed field $k$ of characteristic $0$.

\begin{proposition} \label{prop:twist} 
	Let $A\subset k$ be a $\mathbb{Z}$-finitely generated integrally closed subring and  $\mathcal{X}$ a finite type algebraic stack over $A$ with finite diagonal. Then the map of sets $$\pi_0(\mathcal{X}(A))\to \pi_0(\mathcal{X}(k))$$ has finite fibres. 
\end{proposition}
\begin{proof}
 	Let $x \in \mathcal{X}(A)$. To prove the result it suffices to show that the set
	\begin{equation} \label{eqn:twist}
		\pi_0(\{ y \in \mathcal{X}(A) : y_k \cong x_k\}) \quad \text{is finite}.
	\end{equation}
	(Here $x_k,y_k$ denote the image of $x,y$ in $\mathcal{X}(k)$, respectively.)
	We claim that  the map of sets $$ \pi_0(\mathcal X(A))\to \pi_0(\mathcal X(K)) $$ is injective, where $K=\mathrm{Frac}(A)$. Indeed, for all $y \in \mathcal X(A)$ with $x_k\cong y_k$ the scheme $\Isom_{A}(x,y)\to \Spec A$ is finite as the diagonal of $\mathcal{X}$ is finite. Thus, as $\Spec A$ is integral and normal, the closure of a generic section of $\Isom_{A}(x,y)\to \Spec A$  is a section  \cite[Tag 0AB1]{stacks-project}.	
	
	Therefore to prove \eqref{eqn:twist}, we are free to replace $\Spec A$ by a dense open subset.
	In particular, as the inertia group scheme $I_x$ is finite over $A$ and $K$
	has characteristic $0$, on changing $A$ we may assume that $I_x$ is finite \'etale
	over $A$.
	
	To prove \eqref{eqn:twist}, for $y \in \mathcal{X}(A)$ with $x_k \cong y_k$
	we shall show that the  map
	\begin{equation} \label{eqn:torsor}
		I_x \times_A \Isom_{A}(x,y) \to \Isom_{A}(x,y)
	\end{equation}
	makes $\Isom_{A}(x,y)$ into  an $I_x$-torsor over $A$. As $I_x$ is finite \'etale and
	$$I_x \times_A \Isom_{A}(x,y) \to \Isom_{A}(x,y) \times_A \Isom_{A}(x,y) $$
	is an isomorphism,  it suffices to show that $\Isom_{A}(x,y)$ is faithfully flat over $A$  \cite[Prop.~III.4.1]{MilneEC}.
	
	By \cite[Ex.~5.1.15]{Liu2}, it suffices to show that degree 
	of each fibre of $\Isom_{A}(x,y) \to \Spec A$ is constant
	(viewed as a finite scheme over a field).
	To do so, choose a finite field extension $L$ of $K$ such that $y_L \cong x_L$
	and let $B$ be the normalisation of $A$ in $L$. 
	Then,   since the $A$-scheme $\Isom_{A}(x,y)$ has an $L$-point and $B$ is integral and normal, the $A$-scheme $\Isom_{A}(x,y)$  has  a $B$-point. 
	This implies that $I_{x,B} \cong \Isom_{B}(x,y)$ over $B$, so that 
	$\Isom_{B}(x,y) \to \Spec B$ is finite \'etale. But then for a finite \'etale morphism,
	the degree of each fibre is constant, whence it easily follows that the same holds
	over $A$. This shows that $\Isom_{A}(x,y)$ is  an $I_x$-torsor over $A$ (and that $\Isom_A(x,y)$
	is even finite \'etale over $A$).
	
	We now prove \eqref{eqn:twist}. By a standard argument \cite[p.~134]{MilneEC},
	for $y_1,y_2 \in \mathcal{X}(A)$, if $\Isom_{A}(x,y_1) \cong \Isom_{A}(x,y_2)$ as $I_x$-torsors then $y_1 \cong y_2$. 
	Therefore, the   set of  $A$-isomorphism classes of $y$ in $\mathcal{X}(A)$ with $y_k\cong x_k$ in $\mathcal{X}(k)$ is a subset of $\HH^1(A, I_x)$. Since $\Spec A$ is an arithmetic scheme and $I_x\to \Spec A$ is a finite \'etale group scheme, the latter set is finite by Lemma \ref{lem:H^1_finite}. The result follows.
\end{proof}

\begin{corollary} \label{cor:twists}
Let $A\subset k$ be a $\mathbb{Z}$-finitely generated integrally closed subring with fraction field $K=\mathrm{Frac}(A)$, and let $\mathcal{X}$ be a finite type  algebraic stack over $A$ with finite diagonal. Then the following statements are equivalent.
\begin{enumerate}
\item The set $\mathrm{Im}[ \pi_0(\mathcal{X}(A)) \to \pi_0(\mathcal{X}(k))]$ is finite. 
\item The set $\mathrm{Im}[ \pi_0(\mathcal{X}(A)) \to \pi_0(\mathcal{X}(K))]$ is finite.
\item The set $\pi_0(\mathcal{X}(A))$ is finite.
\end{enumerate}
\end{corollary}
\begin{proof}
	$(3) \Rightarrow (2) \Rightarrow (1)$ is clear. Then $(1) \Rightarrow (3)$
	follows from Proposition~\ref{prop:twist}.
\end{proof}

We also rephrase this property in terms of arithmetic hyperbolicity.

\begin{theorem}[Twisting lemma]\label{thm:equivalences}
	Let $X$ be a finite type  stack   over $k$ with finite diagonal. The following statements are equivalent.
	\begin{enumerate}
		\item The stack $X$ is arithmetically hyperbolic over $k$. 
		\item 	For   all $\ZZ$-finitely generated   subrings $A\subset k$ with fraction field $K$  and all     models $\mathcal X\to \Spec A$ for $X$ over $A$, the  set   $$\Im[\pi_0(\mathcal X(A))\to \pi_0(\mathcal X(K)) ]$$ is finite.
		\item For all $\ZZ$-finitely generated integrally closed subrings $A\subset k$  and all     models $\mathcal X\to \Spec A$ for $X$ over $A$ with $\mathcal{X}$ an algebraic stack  over $A$ with finite diagonal, the  set  $ \pi_0(\mathcal X(A)) $ of isomorphism classes of $A$-integral points on $\mathcal{X}$ is finite.
	\end{enumerate}
\end{theorem}
\begin{proof}
$(2) \Rightarrow (1)$ is clear. Next $(1) \Rightarrow (3)$ follows from 
	Corollary \ref{cor:twists} and Lemma \ref{lem:arithm_hyp}.
	It thus remains to show $(3) \Rightarrow (2)$. Let $A \subset A'\subset k$ be a $\ZZ$-finitely generated subring such that $\mathcal{X}_{A'}$ is an algebraic stack with finite diagonal. This exists as  ``having finite diagonal'' spreads out \cite[Prop.~B.3]{Rydh2}. Therefore, by $(3)$,   the set $\pi_0(\mathcal{X}(A'))$ is finite. Thus the set  $\mathrm{Im}[ \pi_0(\mathcal{X}(A)) \to \pi_0(\mathcal{X}(k))] \subset \mathrm{Im}[ \pi_0(\mathcal{X}(A')) \to \pi_0(\mathcal{X}(k))]$ is finite.
\end{proof}

\begin{remark} \label{remark:M_1}
The conclusion of Theorem \ref{thm:equivalences} can fail for  separated non-Deligne--Mumford   stacks over $k$.  Indeed, let $E$ be an elliptic curve over $\QQ$ with finite Mordell--Weil group, $\mathcal{E}$ its N\'eron model  and $n$ the product of primes of bad reduction of $E$. Then $\mathcal{E}$ has infinitely many pairwise non-isomorphic twists over $A:=\mathbb{Z}[1/n]$, by the footnote on \cite[p.~241]{Maz86}; thus $$\Im[\pi_0(B\mathcal{E}(A))\to \pi_0(B\mathcal{E}(\mathbb{Q}))] = \Im[\mathrm{H}^1(\Spec A, \mathcal{E}) \to \mathrm{H}^1(\Spec \mathbb{Q},\mathcal{E}_{\mathbb{Q}})]$$ is infinite. 
In particular part (2) of Theorem \ref{thm:equivalences} does not hold in this case, even though $BE$ is arithmetically hyperbolic for any elliptic curve $E$ over $k$ (Lemma \ref{lem:G_gerbes_are_AH}). (To see that $BE\to \Spec k$ is separated, note that $E=\Spec k \times_{BE} \Spec k$ is proper over $k$.)

This also shows that the smooth arithmetically hyperbolic finitely presented   separated stack $\mathcal{M}_1$ of smooth proper genus one curves fails Theorem \ref{thm:equivalences}. (For $S$ a scheme, every object of $\mathcal{M}_1(S)$ is a smooth proper morphism $f:X\to S$  of algebraic spaces whose geometric fibres are smooth proper connected curves of genus one.)  
\end{remark}

\begin{remark}  
Let $\mathcal{G}$ be an affine finite type group scheme over $\ZZ$. Note that the finitely presented algebraic stack $B\mathcal{G}_k$ is arithmetically hyperbolic over $k$ for ``trivial'' reasons (Lemma \ref{lem:G_gerbes_are_AH}).  We expect that the integral points on $B\mathcal{G}$ satisfy a stronger finiteness property. Indeed, it seems reasonable to suspect that, for all finitely generated subrings $A\subset k$ with $K:=\mathrm{Frac}(A)$, the set  $$\mathrm{Im}[\mathrm{H}^1(A,\mathcal{G})\to \mathrm{H}^1(K, \mathcal{G}_K)]$$ is finite; see \cite{Rapinchucks, Rapinchucks2, GilleMoretBailly, JL2} for related results. For instance, by \cite[Prop.~5.1]{GilleMoretBailly}, this finiteness holds  if $\dim A=1$.  Also, in \cite[Thm.~9.1.(i)]{Rapinchucks2}, this finiteness result is proven under the assumption that $\dim A = 2$ and $\mathcal{G}$ is the Chevalley group of type $G_2$ (and also in some other cases). 
\end{remark}

\begin{remark} 
Many natural moduli problems over $\QQ$ are finitely presented separated Deligne--Mumford stacks over $\QQ$. However, the natural model for such a stack over $\ZZ$ might not be Deligne--Mumford nor separated. For instance, the stack of smooth plane cubic curves $\mathcal{C}_{(3;1)}$ is a finite type algebraic stack with finite diagonal which is Deligne--Mumford over $\ZZ[1/3]$, but not  over $\ZZ$.
\end{remark}

\section{Chevalley--Weil for algebraic stacks} \label{sec:CW}
The classical theorem of Chevalley--Weil \cite{CW} implies that, if $f:X\to Y$ is a finite \'etale morphism of  algebraic varieties over $\Qbar$, then $X$ is arithmetically hyperbolic if and only if $Y$ is arithmetically hyperbolic. This theorem can be considered as an arithmetic analogue of the similar statement for Brody hyperbolic varieties: if $X\to Y$ is a finite \'etale morphism of algebraic varieties over $\CC$, then $X$ is Brody hyperbolic if and only if $Y$ is Brody hyperbolic. 

We generalize this consequence of the classical Chevalley--Weil theorem to  proper \'etale morphisms of algebraic stacks which may not be representable (e.g.~gerbes), using our version of Hermite--Minkowski. The precise statement (also stated as Theorem \ref{thm:chev_weil_intro}) reads as follows.

\begin{theorem}\label{thm:chev_weil}  
	Let $f:X\to Y$ be a proper \'etale surjective morphism of finitely presented algebraic stacks over an algebraically closed field $k$ of characteristic $0$. Then $X$ is arithmetically hyperbolic over $k$ if and only if $Y$ is arithmetically hyperbolic over $k$.
\end{theorem}
\begin{proof}   
Note that $X\to Y$ is   quasi-finite. Therefore, if $Y$ is arithmetically hyperbolic, then it follows from Proposition \ref{prop:qf}  that $X$ is arithmetically hyperbolic. 

Let us now assume that  $X$ is arithmetically hyperbolic. By Lemma \ref{lem:reducedness}, we may and do assume that $Y$ (and thus $X$) is reduced. Moreover, by Lemma   \ref{lem:irreducible}, we may and do assume that $Y$ is irreducible, hence integral. 
 
	Let $A\subset k$ be a smooth finitely generated $\mathbb Z$-algebra and let $\mathcal X\to \mathcal Y$ be a proper \'etale  morphism of finitely presented integral algebraic stacks over $A$ which is isomorphic to $X\to Y$ after base-change to $k$ and for which the factorial of the inertia degree of every point of  $X$ over $Y$ is invertible on $A$; such data exists because proper \'etale surjective morphisms spread out  \cite[Prop.~B.3]{Rydh2} and the relative inertia stack is finite \'etale (Lemma \ref{lem:inertia}). By Lemma \ref{lem:smooth} it suffices to show that $\Im[\pi_0(\mathcal Y(A))\to \pi_0(\mathcal Y(k))]$ is finite. 	 
	For $\Spec A\to \mathcal Y$ an $A$-point, consider the Cartesian diagram 
	\[
	\xymatrix{B \ar[rr] \ar[d]_{\textrm{proper \'etale surjective}} &  & \mathcal X \ar[d]^{\textrm{proper \'etale surjective}} \\ \Spec A \ar[rr] & & \mathcal Y}
	\]
	By Lemma \ref{lem:constant} we have
	$
	\deg(B/A) = \deg(X/Y).
	$
%
	By Lemma \ref{lem:inertia}, the morphism $I_{\mathcal{X}/\mathcal{Y}} \to \mathcal{X}$
	is finite \'etale. Moreover, by \cite[Tag~06PQ]{stacks-project},
	the relative inertia stack $I_{B/A} \to B$
	is the pull-back of $I_{\mathcal{X}/\mathcal{Y}} \to \mathcal{X}$, hence
	the inertia degree of each point of $B$ over $A$ is less than the maximum of
	the 	inertia degrees of the points of $\mathcal{X}$ over $\mathcal{Y}$.
	As $X$ and $Y$ are fixed, it follows from Hermite--Minkowski for stacks (Theorem \ref{Cor:HM}) that the set of $A$-isomorphism classes of   algebraic stacks $B$ appearing as the pull-back of an $A$-point of $\mathcal Y$ is finite. 
 Therefore, we can simultaneously ``trivialize'' every $B$ appearing from this construction. That is, there is a $\ZZ$-finitely generated subring $A'\subset k$ containing $A$   such that, for all $B$ appearing from the above construction, there is a morphism $\Spec A'\to B$. We fix for each $B$ such a morphism $\Spec A'\to B$.  In a diagram, the situation looks as follows:
 \[
	\xymatrix{& & B \ar[rr] \ar[d] &  & \mathcal X \ar[d]  \\ \Spec A'  \ar[urr] \ar[rr] & & \Spec A \ar[rr] & & \mathcal Y.}
 \]
 
 We now show that $\Im[\pi_0(\mathcal Y(A))\to \pi_0(\mathcal Y(k))]$ is finite. Firstly, 
by what we have shown above,  the image of the natural map $\pi_0(\mathcal X(A')) \to \pi_0(\mathcal Y(A'))$  contains the subset $\mathrm{Im}[\pi_0(\mathcal{Y}(A))\to \pi_0(\mathcal{Y}(A'))]$, i.e., every $A$-point of $\mathcal{Y}$ when viewed as an $A'$-point comes from some $A'$-point on $\mathcal{X}$. However, since $X$ is arithmetically hyperbolic, the set $\Im[\pi_0(\mathcal X(A'))\to \pi_0(\mathcal X(k))]$ is finite. It follows that  the image of the natural map of sets $\Im[\pi_0(\mathcal{X}(A'))\to \pi_0(\mathcal{X}(k))]\to \Im[\pi_0(\mathcal Y(A'))\to \pi_0(\mathcal Y(k))]$  is also finite. Since the image of this map contains $\Im[\pi_0(\mathcal Y(A))\to \pi_0(\mathcal Y(k))]$ as a subset (as was shown  above), the latter set is also finite and the result follows.
\end{proof}

\begin{proof}[Proof of Theorem \ref{thm:cv_intro}]
 Let $A\subset \mathbb C$ be an integrally closed   $\ZZ$-finitely generated subring, and let $Y$ be a finite type separated Deligne--Mumford   stack over $A$. Let $X\to Y$ be a  morphism such that $X_{\CC}\to Y_{\CC}$ is proper \'etale surjective. Suppose that for  every $\ZZ$-finitely generated  subring $A \subset B \subset \CC$ which is \'etale over $A$, the groupoid $X(B)$ is finite. Take $k$ to be the algebraic closure of the field of fractions of $A$.

Let $A \subset A' \subset k$ be a $\ZZ$-finitely generated  subring which is smooth over $A$. As $A \subset A'$ is generically finite, there exist a $\ZZ$-finitely generated subring $A' \subset B \subset k$  which is \'etale over $A$. By hypothesis $X(B)$ is finite. But then
$$\Im[\pi_0(X(A'))\to \pi_0(X(k))]\subset  \Im[\pi_0(X(B))\to \pi_0(X(k))]$$ 
is also finite, hence $X_k$ is arithmetically hyperbolic over $k$ by Lemma \ref{lem:smooth}.
 
 By fpqc descent, the morphism $X_k\to Y_k$ is proper \'etale surjective, so $Y_k$ is arithmetically hyperbolic over $k$ by our Chevalley--Weil theorem (Theorem \ref{thm:chev_weil}). Now, since $Y_k$ is arithmetically hyperbolic over $k$ and $Y$ is a finite type separated Deligne--Mumford   stack  over $A$, it follows from the twisting lemma (Theorem \ref{thm:equivalences}) that the groupoid    $Y(A)$ is finite.
\end{proof}

\section{Applications} \label{sec:appplications}
We now give applications of our results to certain surfaces of general type and  prove a transcendental criterion for arithmetic hyperbolicity.

\subsection{Application to the moduli of surfaces of general type}\label{section:gt}
Let $p$ and $q$ be integers.  Let $\mathcal{S}_{p,q}$ be the stack of canonically polarized surfaces $X$ with $p_g(X) = p$ and $q(X) = q$. Thus, for $S$ a scheme, the objects of the groupoid $\mathcal{S}_{p,q}(S)$ are smooth proper morphisms $X\to S$ of schemes whose geometric fibres $X_s$ are smooth projective connected (minimal) surfaces with ample canonical bundle such that $p_g(X_s) = p$ and $q(X_s) = q$. By Matsusaka--Mumford \cite{MatMum}, the stack $\mathcal{S}_{p,q}$ is a locally finite type separated algebraic stack over $\ZZ$. Moreover,    ``boundedness'' for canonically polarised surfaces  with fixed $p_g$ and $q$ implies that $\mathcal{S}_{p,q}$ is of finite type over $\ZZ$; see \cite[Thm.~1.7]{KollarQuotients}. Finally, as $\mathcal{S}_{p,q}$ parametrizes proper varieties with a (canonical) polarization, the diagonal of $\mathcal{S}_{p,q}$ is affine (cf. the proofs of \cite[Lemma~2.1]{JLFano} and \cite[Lemma~2.4]{JLFano}). 

It seems reasonable to suspect that $\mathcal{S}_{p,q,\CC}$ is arithmetically hyperbolic (cf.~\cite[Conjecture~1.1]{Jav15}). Indeed, its subvarieties are of log-general type by a theorem of Campana--Paun \cite{CP}. Moreover, its subvarieties are Brody hyperbolic \cite{VZ} and even Kobayashi hyperbolic \cite{ToYeung, Schumacher}. Also, the stack $\mathcal{S}_{p,q,\CC}$ satisfies a ``function field'' analogue of arithmetic hyperbolicity by a theorem of Kov\'acs--Lieblich \cite{KovacsLieblich}. Thus, in light of Lang--Vojta's conjectures  \cite[\S0.3]{Abr}, ignoring stacky issues, it seems reasonable to suspect that $\mathcal{S}_{p,q,\CC}$ is arithmetically hyperbolic over $\CC$. Our next result gives a modest contribution towards this expectation, and illustrates how one can use our results  by arguing directly on the moduli stack.

\begin{theorem}\label{thm:ar_hyp_gt} For $q\geq 4$, the stack $\mathcal{S}_{2q-4,q,\CC}$ is arithmetically hyperbolic over $\CC$.
\end{theorem}
\begin{proof}  For $g\geq 2$, let $\mathcal{M}_g$ be the stack of smooth proper curves of genus $g$ over $\ZZ$. Let $f:\mathcal{M}_{2,\CC} \times \mathcal{M}_{q-2,\CC} \to \mathcal{S}_{2q-4,q,\CC}$   be the morphism which associates to a pair $(X,Y)$ in $\mathcal{M}_{2,\CC}\times\mathcal{M}_{q-2,\CC}$ the object $X\times Y$ in $\mathcal{S}_{2q-4,q, \CC}$. Note that $f$ is a well-defined morphism of algebraic stacks over $\CC$. This morphism is surjective by  the  classification   of minimal surfaces of general type $X$ with $p_g(X) = 2q(X) - 4$  over the complex numbers (see Beauville's theorem   in the appendix to \cite{DebarreClass}). In particular, since $\mathcal{M}_{2,\CC}\times \mathcal{M}_{q-2,\CC}$ is connected, the algebraic stack $\mathcal{S}_{2q-4,q,\CC}$ is connected. Note that, by \cite[Thm.~1.1]{Bhatt}
and the connectedness of $\mathcal{S}_{2q-4,q,\CC}$, this morphism is finite \'etale. For every $g\geq 2$, the stack $\mathcal{M}_{g,\CC}$ is arithmetically hyperbolic over $\CC$ (Example \ref{ex:Shaf_Conj}), so that the stack $\mathcal{M}_{2,\CC}\times \mathcal{M}_{q-2,\CC}$ is arithmetically hyperbolic over $\CC$ (Lemma \ref{lem:formal_prop}). Therefore,  as $f:\mathcal{M}_{2,\CC} \times \mathcal{M}_{q-2,\CC} \to \mathcal{S}_{2q-4,q,\CC}$ is finite \'etale, the result follows from the Chevalley--Weil theorem (Theorem \ref{thm:chev_weil}).
\end{proof}
 
\begin{proof}[Proof of Theorem \ref{thm:shaf_for_gt}] 
Since $\mathcal{S}_{2q-4,q}$ is a finite type separated algebraic stack with affine diagonal over $\ZZ$ and $\mathcal{S}_{2q-4,q,\CC}$ is arithmetically hyperbolic over $\CC$ (Theorem \ref{thm:ar_hyp_gt}), it follows from the twisting lemma (Theorem \ref{thm:equivalences}) that $\pi_0(\mathcal{S}_{2q-4,q}(A))$ is finite.  
\end{proof}

\subsection{A transcendental criterion}\label{section:examples} In this section we use Faltings's   finiteness theorem (formerly the Shafarevich conjecture for principally polarized abelian varieties)  and Borel's algebraization theorem to prove a ``transcendental'' criterion for arithmetic hyperbolicity (Theorem \ref{thm:criterion}).  We view this criterion as a confirmation of Lang's philosophy that complex analytic hyperbolicity should have arithmetic consequences.
 
For an integer $g$, let $\mathcal A_g$ be the stack over $\mathbb Z$ of $g$-dimensional principally polarized abelian schemes. We recall some properties of $\mathcal{A}_g$ proven, for instance, in    \cite{MoretBailly, OlssonAbVar}.  The stack $\mathcal{A}_g$ is a smooth finite type separated Deligne--Mumford   stack over $\mathbb{Z}$ whose coarse space is a quasi-projective scheme over $\ZZ$.  For $n\geq 1$, let $\mathcal{A}_{g}^{[n]}$ be the stack over $\ZZ[1/n]$ of $g$-dimensional principally polarized abelian schemes with a full level $n$-structure. Note that the forgetful functor $\mathcal{A}_{g}^{[n]} \to \mathcal{A}_{g,\ZZ[1/n]}$ is finite \'etale. Moreover, for $n\geq 3$, the stack $\mathcal{A}_{g}^{[n]}$ is (representable by) a quasi-projective scheme over $\ZZ[1/n]$.

If $X$ is a locally finite type scheme over $\CC$, we let $X^{\an}$ be the associated complex analytic space \cite[Expos\'e~XII]{SGA1}. 
Our goal in this section is provide a precise  interplay between the ``analytic'' hyperbolicity of $\mathcal{A}_{g,\CC}$ (i.e., $\mathcal{A}_{g,\CC}^{\an}$ is hyperbolically embedded in its Baily--Borel compactification) and the arithmetic hyperbolicity of $\mathcal{A}_{g,\CC}$ (as proven by Faltings). We start with an analytic property of $\mathcal{A}_{g,\CC}^{\an}$ (which is also studied in \cite{JKuch}).

 	\begin{lemma}[Borel's algebraization theorem]\label{lem:Borel}
 	Let $X$ be a finite type reduced  scheme  over $\CC$, and let $\varphi:X^{\an}\to \mathcal{A}_{g,\CC}^{[3], \an}$ be a morphism. Then $\varphi$ is algebraic, i.e., there is a unique morphism of schemes $f:X\to \mathcal{A}_{g,\CC}^{[3]}$ such that $f^{\an} =\varphi$.
 	\end{lemma}
 	\begin{proof}  The uniqueness of $f$ is clear. Since $\mathcal{A}_{g,\CC}^{[3],\an}$ is a locally symmetric variety, the result follows from   Borel's theorem \cite[Thm.~3.10]{Borel1972} (see also \cite[Thm.~5.1]{DeligneK3}). 
 	\end{proof}

 	We now prove a  generalization of Borel's algebraization theorem to stacks. To state this result, for $X$ a finitely presented algebraic stack   over $\CC$, we let $X^{\an}$ be the associated complex-analytic stack; see \cite[\S6.1]{PortaYu}  for a definition of the stack $X^{\an}$.
 	\begin{proposition}[Stacky Borel algebraization]\label{lem:Borel2}  
 	Let $X$ be a finitely presented  reduced   algebraic stack  over $\CC$, and let $\varphi:X^{\an}\to \mathcal{A}_{g,\CC}^{[3], \an}$ be a morphism. Then $\varphi$ is algebraic, i.e., there is a unique morphism of stacks $f:X\to \mathcal{A}_{g,\CC}^{[3]}$ such that $f^{\an} =\varphi$.  
 	\end{proposition}
 	\begin{proof}
 	As $\mathcal{A}_{g,\CC}^{[3]}$ is a scheme, the functor $\Hom(\cdot,\mathcal{A}_{g,\CC}^{[3]})$ is a sheaf for the fppf topology on   stacks over $\CC$. 
 	Let $P\to X$ be a smooth surjective morphism with $P$ a finite type reduced scheme over $\CC$. Then, by Borel's algebraization theorem (Lemma \ref{lem:Borel}), the morphism  $P^{\an}\to \mathcal{A}_{g,\CC}^{[3],\an}$ is the analytification of a unique morphism $P\to \mathcal{A}_{g,\CC}^{[3]}$. Similarly, the morphism $(P\times_X P)^{\an} = P^{\an}\times_{X^{\an}} P^{\an} \to \mathcal{A}_{g,\CC}^{[3],\an}$ is the analytification of a unique morphism $P\times_X P\to \mathcal{A}_{g,\CC}^{[3]}$. Thus, 
 	by the sheaf property of $\mathrm{Hom}(\cdot,\mathcal{A}_{g,\CC}^{[3]})$, 
 	we conclude that the morphism $\varphi: X^{\an}\to \mathcal{A}_{g,\CC}^{[3],\an}$ is the analytification of a unique morphism  $f:X\to \mathcal{A}_{g,\CC}^{[3]}$. 
 	\end{proof}

 	 \begin{theorem}[Transcendental criterion]\label{thm:criterion}
 	 	Let $X$ be a   finitely presented algebraic stack over $\mathbb C$. If there exist a finitely presented algebraic stack $Y$,  a proper \'etale morphism $Y\to X$ and  a quasi-finite holomorphic map $Y^{\textrm{an}} \to \mathcal A_{g,\mathbb C}^{\textrm{an}}$, then $X$ is arithmetically hyperbolic over $\CC$.
 	 \end{theorem}
 	 \begin{proof} We may and do assume that $X$ is reduced (Lemma \ref{lem:reducedness}).  In particular, $Y$ is  a reduced finitely presented algebraic stack over $\CC$. Moreover, let $Y' =  Y^{\an}\times_{\mathcal{A}_{g,\CC}^{\an}} \mathcal{A}_{g,\CC}^{[3],\an}$. Then the natural holomorphic map $Y'\to Y^{\an}$ is finite \'etale. Therefore, by the stacky version of  Riemann's existence theorem \cite[Thm.~20.4]{Noohi2}, there is a reduced finitely presented algebraic stack $Z$ over $\CC$, a finite \'etale morphism $Z\to Y$, and an isomorphism $Z^{\an} \cong Y'$ over $Y^{\an}$.  
 	 Note that $Z^{\an}\to \mathcal{A}_{g,\CC}^{[3],\an}$ is a quasi-finite holomorphic map.  	 	By stacky Borel algebraization (Proposition \ref{lem:Borel2}), there is a (quasi-finite) morphism $Z\to \mathcal{A}_{g,\CC}^{[3]}$.

 	 	By Faltings's finiteness theorem, the stack $\mathcal{A}_{g,\CC}$ is arithmetically hyperbolic over $\CC$; see \cite{FaltingsComplements} (which builds on \cite{Faltings2}). Therefore, the scheme $\mathcal{A}_{g,\CC}^{[3]}$ is arithmetically hyperbolic by Proposition \ref{prop:qf}.
 Since $\mathcal A_{g,\mathbb C}^{[3]}$ is arithmetically hyperbolic over $\CC$, it   follows from Proposition \ref{prop:qf} that
 	 	   $Z$ is arithmetically hyperbolic over $\CC$.
 	 	   By  the Chevalley--Weil theorem (Theorem \ref{thm:chev_weil}), as $Z\to Y\to X$ is proper \'etale, we conclude that $X$ is arithmetically hyperbolic over $\CC$.
 	 \end{proof}

 \subsection{Application to cubic threefolds} \label{sec:cubic_threefolds}
 Being able to pass to a proper \'etale cover in Theorem \ref{thm:criterion} is very natural for applications. Such  covers of moduli stacks often naturally arise by adding level structure to the objects parametrized by the stack, where $Y$ is usually even a scheme. In practice, morphisms to $\mathcal A_{g,\mathbb C}^{\textrm{an}}$ arise via period maps, and such a morphism being quasi-finite translates to a (local) Torelli theorem.

  We give an application of this type, which is a new proof of the arithmetic hyperbolicity of the moduli of smooth cubic threefolds \cite[Thm.~1.1]{JL}. This proof is much simpler than the proof given in \cite{JL}, as it avoids the need to create an algebraic theory of intermediate Jacobians, i.e.~it avoids the need to construct the intermediate jacobian $\mathcal{C}_{(3;3),\CC} \to \mathcal{A}_{5,\CC}$ as a morphism of algebraic stacks over $\CC$ (hence also avoids the need for an arithmetic theory of intermediate Jacobians, given by a morphism of stacks $\mathcal{C}_{(3;3),\QQ} \to \mathcal{A}_{5,\QQ}$).

\begin{theorem}
	The stack of smooth cubic threefolds $\mathcal{C}_{(3;3),\CC}$
	 is arithmetically hyperbolic over $\CC$.
\end{theorem}
\begin{proof}  Let $X:=\mathcal{C}_{(3;3),\CC}$. Since $X$ is uniformisable by an affine scheme \cite{JLLevel}, there is  an affine variety $Y$ over $\CC$ and a finite \'etale morphism $Y\to X$. Let $\mathbb{V}$ be the polarized variation of Hodge structures on $Y$ associated to the pull-back of the universal family $\mathcal{U}\to X$ along $Y\to X$. Let $Y^{\an}\to \mathcal{A}_{5,\CC}^{\an}$ be the associated period map. By infinitesimal Torelli for smooth cubic threefolds, the latter morphism is injective on tangent spaces. In particular, it has finite fibres   (see for instance \cite[Thm.~2.8]{JL}).  
Therefore, the arithmetic hyperbolicity of $X$ over $\CC$ follows from the transcendental criterion (Theorem \ref{thm:criterion}). 
\end{proof}

\bibliography{refsci}{}

\def\cprime{$'$}
\begin{thebibliography}{10}

\bibitem{Abr}
D.~Abramovich.
\newblock Uniformity of stably integral points on elliptic curves.
\newblock {\em Invent. Math.}, 127(2):307--317, 1997.

\bibitem{ACV}
D.~Abramovich, A.~Corti, and A.~Vistoli.
\newblock Twisted bundles and admissible covers.
\newblock {\em Comm. Algebra}, 31(8):3547--3618, 2003.
\newblock Special issue in honor of Steven L. Kleiman.

\bibitem{AOV}
D.~Abramovich, M.~Olsson, and A.~Vistoli.
\newblock Tame stacks in positive characteristic.
\newblock {\em Ann. Inst. Fourier (Grenoble)}, 58(4):1057--1091, 2008.

\bibitem{Autissier1}
P.~Autissier.
\newblock G\'eom\'etries, points entiers et courbes enti\`eres.
\newblock {\em Ann. Sci. \'Ec. Norm. Sup\'er. (4)}, 42(2):221--239, 2009.

\bibitem{Autissier2}
P.~Autissier.
\newblock Sur la non-densit\'e des points entiers.
\newblock {\em Duke Math. J.}, 158(1):13--27, 2011.

\bibitem{Bhatt}
B.~Bhatt, W.~Ho, Z.~Patakfalvi, and C.~Schnell.
\newblock Moduli of products of stable varieties.
\newblock {\em Compos. Math.}, 149(12):2036--2070, 2013.

\bibitem{vBJK}
R.~van Bommel, A.~Javanpeykar, and L.~Kamenova.
\newblock Boundedness in families with applications to arithmetic
  hyperbolicity.
\newblock {\em arXiv:1907.11225}.

\bibitem{Borel1972}
A.~Borel.
\newblock Some metric properties of arithmetic quotients of symmetric spaces
  and an extension theorem.
\newblock {\em J. Differential Geometry}, 6:543--560, 1972.

\bibitem{CP}
F.~Campana and M.~Paun.
\newblock Orbifold generic semi-positivity: an application to families of
  canonically polarized manifolds.
\newblock {\em Ann. Inst. Fourier (Grenoble)}, 65(2):835--861, 2015.

\bibitem{Rapinchucks}
V.~I. Chernousov, A.~S. Rapinchuk, and I.~A. Rapinchuk.
\newblock On some finiteness properties of algebraic groups over finitely
  generated fields.
\newblock {\em C. R. Math. Acad. Sci. Paris}, 354(9):869--873, 2016.

\bibitem{Rapinchucks2}
V.~I. Chernousov, A.~S. Rapinchuk, and I.~A. Rapinchuk.
\newblock Spinor groups with good reduction.
\newblock {\em Compos. Math.}, 155(3):484--527, 2019.

\bibitem{CW}
V.~Chevalley and A.~Weil.
\newblock Un th\'eor\`eme d'arithm\'etique sur les courbes alg\'ebriques.
\newblock {\em C. R. Acad. Sci., Paris.}, (195):570--572, 1932.

\bibitem{CLZ}
P.~Corvaja, A.~Levin, and U.~Zannier.
\newblock Integral points on threefolds and other varieties.
\newblock {\em Tohoku Math. J. (2)}, 61(4):589--601, 2009.

\bibitem{CorZanAnnals}
P.~Corvaja and U.~Zannier.
\newblock On integral points on surfaces.
\newblock {\em Ann. of Math. (2)}, 160(2):705--726, 2004.

\bibitem{DebarreClass}
O.~Debarre.
\newblock In\'egalit\'es num\'eriques pour les surfaces de type g\'en\'eral.
\newblock {\em Bull. Soc. Math. France}, 110(3):319--346, 1982.
\newblock With an appendix by A. Beauville.

\bibitem{DeligneK3}
P.~Deligne.
\newblock La conjecture de {W}eil pour les surfaces {$K3$}.
\newblock {\em Invent. Math.}, 15:206--226, 1972.

\bibitem{Faltings2}
G.~Faltings.
\newblock Endlichkeitss\"atze f\"ur abelsche {V}ariet\"aten \"uber
  {Z}ahlk\"orpern.
\newblock {\em Invent. Math.}, 73(3):349--366, 1983.

\bibitem{FaltingsComplements}
G.~Faltings.
\newblock Complements to {M}ordell.
\newblock In {\em Rational points ({B}onn, 1983/1984)}, Aspects Math., E6,
  pages 203--227. Vieweg, Braunschweig, 1984.

\bibitem{FaltingsLang}
G.~Faltings.
\newblock The general case of {S}. {L}ang's conjecture.
\newblock In {\em Barsotti {S}ymposium in {A}lgebraic {G}eometry ({A}bano
  {T}erme, 1991)}, volume~15 of {\em Perspect. Math.}, pages 175--182. Academic
  Press, San Diego, CA, 1994.

\bibitem{GilleMoretBailly}
P.~Gille and L.~Moret-Bailly.
\newblock Actions alg\'ebriques de groupes arithm\'etiques.
\newblock In {\em Torsors, \'etale homotopy and applications to rational
  points}, volume 405 of {\em London Math. Soc. Lecture Note Ser.}, pages
  231--249. Cambridge Univ. Press, Cambridge, 2013.

\bibitem{Gir71}
J.~Giraud.
\newblock {\em Cohomologie non abelienne}.
\newblock Springer-Verlag, Berlin-New York, 1971.
\newblock Die Grundlehren der mathematischen Wissenschaften, Band 179.

\bibitem{SGA1}
A.~Grothendieck.
\newblock {\em Rev\^etements \'etales et groupe fondamental (SGA I) {F}asc.
  {II}: {E}xpos\'es 6, 8 \`a 11}, volume 1960/61 of {\em S\'eminaire de
  G\'eom\'etrie Alg\'ebrique}.
\newblock Institut des Hautes \'Etudes Scientifiques, Paris, 1963.

\bibitem{smallness}
S.~Harada and T.~Hiranouchi.
\newblock Smallness of fundamental groups for arithmetic schemes.
\newblock {\em J. Number Theory}, 129(11):2702--2712, 2009.

\bibitem{Jaut}
A.~Javanpeykar.
\newblock Arithmetic hyperbolicity: automorphisms and persistence.
\newblock {\em arXiv:1809.06818}.

\bibitem{Jav15}
A.~Javanpeykar.
\newblock N\'eron models and the arithmetic {S}hafarevich conjecture for
  certain canonically polarized surfaces.
\newblock {\em Bull. Lond. Math. Soc.}, 47(1):55--64, 2015.

\bibitem{JKuch}
A.~Javanpeykar and R.~Kucharczyk.
\newblock Algebraicity of analytic maps to a hyperbolic variety.
\newblock {\em Math. Nachr. 293 (2020), no. 8 (August)}.

\bibitem{JLevin}
A.~Javanpeykar and A.~Levin.
\newblock Urata's theorem in the logarithmic case and applications to integral
  points.
\newblock {\em arXiv:2002.11709}.

\bibitem{JLitt}
A.~Javanpeykar and D.~Litt.
\newblock Integral points on algebraic subvarieties of period domains: from
  number fields to finitely generated fields.
\newblock {\em arXiv:1907.13536}.

\bibitem{JL2}
A.~Javanpeykar and D.~Loughran.
\newblock Good reduction of algebraic groups and flag varieties.
\newblock {\em Arch. Math. (Basel)}, 104(2):133--143, 2015.

\bibitem{JL}
A.~Javanpeykar and D.~Loughran.
\newblock Complete intersections: moduli, {T}orelli, and good reduction.
\newblock {\em Math. Ann.}, 368(3-4):1191--1225, 2017.

\bibitem{JLLevel}
A.~Javanpeykar and D.~Loughran.
\newblock The moduli of smooth hypersurfaces with level structure.
\newblock {\em Manuscripta Math.}, 154(1-2):13--22, 2017.

\bibitem{JLFano}
A.~Javanpeykar and D.~Loughran.
\newblock Good reduction of {F}ano threefolds and sextic surfaces.
\newblock {\em Ann. Sc. Norm. Super. Pisa Cl. Sci. (5)}, 18(2):509--535, 2018.

\bibitem{JLalg2}
A.~Javanpeykar, D.~Loughran, and S.~Mathur.
\newblock Good reduction and cyclic covers.
\newblock {\em arXiv:2009.01831}.

\bibitem{JSZ}
A.~Javanpeykar, R.~Sun, and K.~Zuo.
\newblock The {S}hafarevich conjecture revisited: Finiteness of pointed
  families of polarized varieties.
\newblock {\em arXiv:2005.05933}.

\bibitem{Knutson}
D.~Knutson.
\newblock {\em Algebraic spaces}.
\newblock Lecture Notes in Mathematics, Vol. 203. Springer-Verlag, Berlin-New
  York, 1971.

\bibitem{KollarQuotients}
J.~Koll\'ar.
\newblock Quotient spaces modulo algebraic groups.
\newblock {\em Ann. of Math. (2)}, 145(1):33--79, 1997.

\bibitem{KovacsLieblich}
S.~J. Kov{\'a}cs and M.~Lieblich.
\newblock Erratum for {B}oundedness of families of canonically polarized
  manifolds: a higher dimensional analogue of {S}hafarevich's conjecture.
\newblock {\em Ann. of Math. (2)}, 173(1):585--617, 2011.

\bibitem{Levin}
A.~Levin.
\newblock Generalizations of {S}iegel's and {P}icard's theorems.
\newblock {\em Ann. of Math. (2)}, 170(2):609--655, 2009.

\bibitem{Liu2}
Q.~Liu.
\newblock {\em Algebraic geometry and arithmetic curves}, volume~6 of {\em
  Oxford Graduate Texts in Mathematics}.
\newblock Oxford University Press, Oxford, 2006.
\newblock Translated from the French by Reinie Ern{\'e}, Oxford Science
  Publications.

\bibitem{MatMum}
T.~Matsusaka and D.~Mumford.
\newblock Two fundamental theorems on deformations of polarized varieties.
\newblock {\em Amer. J. Math.}, 86:668--684, 1964.

\bibitem{Maz86}
B.~Mazur.
\newblock Arithmetic on curves.
\newblock {\em Bull. Amer. Math. Soc.}, 14(2):207--259, 1986.

\bibitem{MilneEC}
J.~S. Milne.
\newblock {\em \'{E}tale cohomology}, volume~33 of {\em Princeton Mathematical
  Series}.
\newblock Princeton University Press, Princeton, N.J., 1980.

\bibitem{Mil06}
J.~S. Milne.
\newblock {\em Arithmetic duality theorems}.
\newblock BookSurge, LLC, Charleston, SC, second edition, 2006.

\bibitem{MoretBailly}
L.~Moret-Bailly.
\newblock Pinceaux de vari\'et\'es ab\'eliennes.
\newblock {\em Ast\'erisque}, (129):266, 1985.

\bibitem{Noohi2}
B.~Noohi.
\newblock Foundations of topological stacks {I}.
\newblock {\em arXiv:math/0503247}.

\bibitem{OlssonAbVar}
M.~C. Olsson.
\newblock {\em Compactifying moduli spaces for abelian varieties}, volume 1958
  of {\em Lecture Notes in Mathematics}.
\newblock Springer-Verlag, Berlin, 2008.

\bibitem{Poo17}
B.~Poonen.
\newblock {\em Rational Points on Varieties}, volume 186 of {\em Graduate
  Studies in Mathematics}.
\newblock American Mathematical Society, Providence, RI, 2017.

\bibitem{PortaYu}
M.~Porta and T.~Y. Yu.
\newblock Higher analytic stacks and {GAGA} theorems.
\newblock {\em Adv. Math.}, 302:351--409, 2016.

\bibitem{Romagny}
M.~Romagny.
\newblock Group actions on stacks and applications.
\newblock {\em Michigan Math. J.}, 53(1):209--236, 2005.

\bibitem{Rydh}
D.~Rydh.
\newblock Existence and properties of geometric quotients.
\newblock {\em J. Algebraic Geom.}, 22(4):629--669, 2013.

\bibitem{Rydh2}
D.~Rydh.
\newblock Noetherian approximation of algebraic spaces and stacks.
\newblock {\em J. Algebra}, 422:105--147, 2015.

\bibitem{Schumacher}
G.~Schumacher.
\newblock Curvature properties for moduli of canonically polarized
  manifolds---an analogy to moduli of {C}alabi-{Y}au manifolds.
\newblock {\em C. R. Math. Acad. Sci. Paris}, 352(10):835--840, 2014.

\bibitem{Ser97}
J.-P. Serre.
\newblock {\em Lectures on the {M}ordell-{W}eil theorem}.
\newblock Aspects of Mathematics. Friedr. Vieweg \& Sohn, Braunschweig, third
  edition, 1997.

\bibitem{stacks-project}
The {Stacks Project Authors}.
\newblock \emph{{S}tacks {P}roject}.
\newblock http://stacks.math.columbia.edu, 2018.

\bibitem{ToYeung}
W.-K. To and S.-K. Yeung.
\newblock Finsler metrics and {K}obayashi hyperbolicity of the moduli spaces of
  canonically polarized manifolds.
\newblock {\em Ann. of Math. (2)}, 181(2):547--586, 2015.

\bibitem{UllmoShimura}
E.~Ullmo.
\newblock Points rationnels des vari\'et\'es de {S}himura.
\newblock {\em Int. Math. Res. Not.}, (76):4109--4125, 2004.

\bibitem{VZ}
E.~Viehweg and K.~Zuo.
\newblock On the {B}rody hyperbolicity of moduli spaces for canonically
  polarized manifolds.
\newblock {\em Duke Math. J.}, 118(1):103--150, 2003.

\bibitem{Vistoli}
A.~Vistoli.
\newblock Intersection theory on algebraic stacks and on their moduli spaces.
\newblock {\em Invent. Math.}, 97(3):613--670, 1989.

\bibitem{VojtaSub}
P.~Vojta.
\newblock A refinement of {S}chmidt's subspace theorem.
\newblock {\em Amer. J. Math.}, 111(3):489--518, 1989.

\bibitem{Vojta1}
P.~Vojta.
\newblock Integral points on subvarieties of semiabelian varieties. {I}.
\newblock {\em Invent. Math.}, 126(1):133--181, 1996.

\bibitem{Vojta2}
P.~Vojta.
\newblock Integral points on subvarieties of semiabelian varieties. {II}.
\newblock {\em Amer. J. Math.}, 121(2):283--313, 1999.

\end{thebibliography}
\bibliographystyle{plain}

\end{document}